\documentclass[11pt]{cedram-aif.cls}
\pdfoutput=1

\usepackage{eucal}
\usepackage{graphicx}
\usepackage{color}
\usepackage{amsopn,amssymb}

\definecolor{verydarkblue}{rgb}{0,0,0.4}
\usepackage{hyperref}
\usepackage{mathrsfs}

\usepackage{fancyhdr}

\usepackage{marginnote}

\reversemarginpar

\graphicspath{{figures/}}

\makeatletter%
\def\@commafont{\check@mathfonts
    \fontsize\sf@size\z@\selectfont}
\DeclareRobustCommand{\cb}[1]{{%
\setbox\z@\hbox{#1}%
\ifdim\dp\z@<.1\ht\z@\ooalign{\unhbox\z@\crcr\hidewidth\lower.3ex\hbox{\@commafont,}\hidewidth}%
\else\ooalign{\unhbox\z@\crcr\hidewidth\raise.5ex\hbox{\@commafont`}\hidewidth}%
\fi}}%
\makeatother

\newcommand{\grad}{\nabla}

\DeclareMathOperator{\im}{Im}

\renewcommand{\Im}{\im}

\renewcommand{\tilde}[1]{\widetilde{#1}}

\newcommand{\g}{\mathfrak{g}}

\newcommand{\tec}{Teichm\"uller }
\newcommand{\wep}{Weil-Petersson }

\newcommand{\sU}{\mathcal{U}}
\newcommand{\sT}{\mathcal{T}}

\newcommand{\sM}{\mathcal{M}}

\renewcommand{\leq}{\leqslant}
\renewcommand{\geq}{\geqslant}

\DeclareMathOperator{\ssys}{\mathrm{sys}}

\DeclareMathOperator{\diam}{diam}

\DeclareMathOperator{\dist}{dist}

\DeclareMathOperator{\Teich}{Teich}
\DeclareMathOperator{\Te}{Teich}
\DeclareMathOperator{\InR}{InRad}
\DeclareMathOperator{\Diff}{Diff}

\DeclareMathOperator{\vol}{Vol}
\newcommand{\Mod}{\mbox{\rm Mod}}

\newcommand{\lsys}{\ell_{sys}}
\DeclareMathOperator{\arccosh}{arccosh}
\DeclareMathOperator{\Vol}{Vol}
\DeclareMathOperator{\Ric}{Ric}
\DeclareMathOperator{\arcsinh}{arcsinh}

\usepackage{amsthm}

\newcommand{\param}{{\mathchoice{\mkern1mu\mbox{\raise2.2pt\hbox{$\centerdot$}}\mkern1mu}{\mkern1mu\mbox{\raise2.2pt\hbox{$\centerdot$}}\mkern1mu}{\mkern1.5mu\centerdot\mkern1.5mu}{\mkern1.5mu\centerdot\mkern1.5mu}}}

\numberwithin{equation}{section}

\theoremstyle{plain}

\newtheorem{corollary}{Corollary}
\newtheorem{proposition}{Proposition}

\newtheorem{claim}{Claim}

\theoremstyle{definition}

\newtheorem{question}{Question}

\newtheorem{remark}{Remark}
\theoremstyle{definition}
\newtheorem*{remarksenv}{Remarks}

\begin{document}

\title[Growth]
{Growth of the Weil-Petersson inradius of moduli space}

\author{Yunhui Wu}

\address{Yau Mathematical Sciences Center\\
      Tsinghua University\\
         Beijing, China, 100084\\}
\email{yunhui\_wu@mail.tsinghua.edu.cn}


\begin{abstract}
In this paper we study the systole function along \wep geodesics. We show that the square root of the systole function is uniformly Lipschitz on \tec space endowed with the \wep metric. As an application, we study the growth of the \wep inradius of moduli space of Riemann surfaces of genus $g$ with $n$ punctures as a function of $g$ and $n$. We show that the \wep inradius is comparable to $\sqrt{\ln{g}}$ with respect to $g$, and is comparable to $1$ with respect to $n$. 

Moreover, we also study the asymptotic behavior, as $g$ goes to infinity, of the \wep volumes of geodesic balls of finite radii in \tec space. We show that they behave like $o((\frac{1}{g})^{(3-\epsilon)g})$ as $g\to \infty$, where $\epsilon>0$ is arbitrary.
\end{abstract}


\maketitle

\thispagestyle{fancy}

\section{Introduction}
Let $S_{g,n}$ be a surface of genus $g$ with $n$ punctures with $3g+n\geq4$, and $\Teich(S_{g,n})$ be \tec space of $S_{g,n}$ endowed with the \wep metric. The mapping class group $\Mod(S_{g,n})$ of $S_{g,n}$ acts on $\Teich(S_{g,n})$ by isometries. The moduli space $\sM_{g,n}$ of $S_{g,n}$, endowed with the \wep metric, is realized as the quotient $\Teich(S_{g,n})/\Mod(S_{g,n})$. 

The moduli space $\sM_{g,n}$ is K\"ahler \cite{Ahlfors61}, incomplete \cite{Chu76, Wolpert75} and geodesically complete \cite{Wolpert87}. It has negative sectional curvature \cite{Tromba86, Wolpert86}, strongly negative curvature in the sense of Siu \cite{Schu86}, dual Nakano negative curvature \cite{LSY08} and nonpositive definite Riemannian curvature operator \cite{Wu14}. The \wep metric completion $\overline{\sM}_{g,n}$ of moduli space $\sM_{g,n}$, as a topological space, is the well-known Deligne-Mumford compactification of moduli space obtained by adding stable nodal curves \cite{Masur76}. One may refer to the book \cite{Wolpertbook} for recent developments on the \wep metric.

The asymptotic geometry of $\sM_{g,n}$ as either $g$ or $n$ tends to infinity, has recently become quite active. For example, Brock-Bromberg \cite{BB2014} showed that the shortest \wep closed geodesic in $\sM_{g,0}$ is comparable to $\frac{1}{\sqrt{g}}$. Mirzakhani \cite{Mirz07-j, Mirz10, Mirz13, MirzZ15} studied various aspects of the \wep volume of $\sM_{g,n}$ for large $g$. Together with M. Wolf \cite{WW15}, we studied the $\ell^p$-norm $(1\leq p \leq \infty)$ of the \wep curvature operator of $\sM_{g,n}$ for large $g$. The \wep curvature of $\sM_{g,0}$ for large genus was studied in \cite{Wu15-curvature}. Cavendish-Parlier \cite{CP12} studied the asymptotic behavior of the diameter $\diam(\sM_{g,n})$ of $\sM_{g,n}$. They showed that $\displaystyle\lim_{n\to \infty}\frac{\diam(\sM_{g,n})}{\sqrt{n}}$ is a positive constant. They also showed that for large genus the ratio $\frac{\diam(\sM_{g,n})}{\sqrt{g}}$ is bounded below by a positive constant and above by a constant multiple of $\ln{g}$. For the upper bound, they refined Brock's quasi-isometry of $\Teich(S_{g,n})$ to the pants graph \cite{Brock03}. As far as we know, the asymptotic behavior of $\diam(\sM_{g,n})$ as $g$ tends to infinity is still \textsl{open}. For other related topics, one may refer to \cite{FKM13, GPY11, LX09, Penner92, RT13, ST01, Zograf08} for more details.

Let $\partial \overline{\sM}_{g,n}$ be the boundary of $\overline{\sM}_{g,n}$, which consists of nodal surfaces. Let $\dist_{wp}(\cdot, \cdot)$ be the \wep distance function.  Define the \textsl{inradius} $\InR(\sM_{g,n})$ of $\sM_{g,n}$ as  
\[\InR(\sM_{g,n}):=\max_{X\in \sM_{g,n}}\dist_{wp}(X, \partial \overline{\sM}_{g,n}).\]
The inradius $\InR(\sM_{g,n})$ is the largest radius of geodesic balls (allowed to contain topology) in the interior of $\overline{\sM}_{g,n}$. In this paper, one of our main goals is to study the asymptotic behavior of $\InR(\sM_{g,n})$ either as $g\to \infty$ or $n\to \infty$.
\newline

\noindent \textbf{Notation.} In this paper, we use the notation
$$f_1 \asymp_t f_2$$ 
if there exists a universal constant $C>0$, independent of $t$, such that
$$  \frac{f_2}{C} \leq f_1\leq C f_2.$$

Our first result is
\begin{theorem}\label{mt-g}
For all $n\geq 0$ and $g\geq 2$, we have
\[ \InR (\sM_{g,n}) \asymp_g \sqrt{\ln{g}.}\]
\end{theorem}

We will show that as $g\to \infty$, the inradius $\InR(\sM_{g,n})$ is roughly realized by the family of surfaces constructed by Balacheff-Makover-Parlier in \cite{BMP14} (based on the work of Buser-Sarnak \cite{BS94}), whose injectivity radii grow roughly as $\ln{g}$. We remark here that the method used in the proof of Theorem \ref{mt-g} also shows that $\InR (\sM_{g,[g^a]})\asymp_g \sqrt{\ln{g}}$ for all $a \in (0,1)$. One can see Remark \ref{n(g)} for more details.\\

Our second result is
\begin{theorem}\label{mt-n}
For all $g\geq 0$ and $n\geq 4$, we have 
\[\InR (\sM_{g,n})\asymp_n 1.\]
\end{theorem}

We remark that the method used in the proof of Theorem \ref{mt-n} also gives that $\InR (\sM_{[n^a],n})\asymp_n 1$ for all $a \in (0,1)$. One can see Remark \ref{g(n)} for more details. We will give two different proofs for the lower bound in Theorem \ref{mt-n}, one of which is by applying Theorem \ref{mt-lip}.
 
The difficult parts for Theorem \ref{mt-g} and \ref{mt-n} are the lower bounds, which rely on studying the systole function along \wep geodesics. \\

For any $X\in \Teich(S_{g,n})$, we refer to the length of a shortest essential simple closed geodesic in X as the \emph{systole} of X and denote it by $\ell_{sys}(X)$. The systole function $\lsys(\cdot):\Teich(S_{g,n})\to \mathbb{R}^+$ is continuous, but not smooth as corners appear when it is realized by multiple essential isotopy classes of simple closed curves. However, it is a topological Morse function and its critical points can be characterized. One may refer to \cite{Akr03, Gen15, Sch93} for more details. The lower bounds in Theorems \ref{mt-g} and \ref{mt-n} will be established by using the following theorem, which gives a uniform lower bound for the \wep distance in terms of systole functions. 
\begin{theorem}\label{mt-lip}
There exists a universal constant $K>0$, independent of $g$ and $n$, such that for all $X,Y\in \Teich(S_{g,n})$,
\[|\sqrt{\ell_{sys}(X)}-\sqrt{\ell_{sys}(Y)}|\leq K\dist_{wp}(X,Y).\]
\end{theorem}

To the best of our knowledge, Theorem 1.3 is the first study of the systole function along Weil-Petersson geodesics, addressing a line of inquiry that Wolpert raised in \cite[page 274]{Wolpert-per}: \textsl{determine the behaviors of the systole function along \wep geodesics}. For the limits of relative systolic curves along a \wep geodesic ray in Thurston's projective measured lamination space, one may see \cite{BMM10, BMM11, BM15, Hame15} for more details. 

The strategy for establishing Theorem \ref{mt-lip} is to bound the \wep norm of the gradient $\grad \ell_{\alpha}^{\frac{1}{2}}(X)$ from above by a universal constant, independent of $g$ and $n$, when $\alpha$ is an essential simple closed curve in $X$ which realizes the systole of $X$. In order to do this, first by applying the real analyticity of the \wep metric \cite{Ahlfors61} and the convexity of geodesic length function along \wep geodesics \cite{Wolpert87, Wolf12}, we make a thin-thick decomposition for the \wep geodesic $\g(X,Y)\subset \Teich(S_{g,n})$ connecting $X$ and $Y$ such that we can differentiate $\lsys(\cdot)$ along the geodesic $\g(X,Y)$ in some sense (see Lemma \ref{key-l-2} in Section \ref{one-l}). Then, for the thin part of $\g(X,Y)$ we use a result, due to Wolpert in \cite{Wolpert08} (see Lemma 3.16 in \cite{Wolpert08} or Lemma \ref{short} in Section \ref{sys-wg}), to get a uniform upper bound for the \wep norm of the gradient $\grad \ell_{\alpha}^{\frac{1}{2}}(X)$. For the thick part of $\g(X,Y)$ (here the injectivity radius of some hyperbolic surface, which is a point on $\g(X,Y)$, could be arbitrarily large \cite{BS94}), we apply a special case of a formula of Riera \cite{Rie05} (see Equation (\ref{Rie-f}) in Section \ref{sys-wg}) and some two-dimensional hyperbolic geometry theory to provide a uniform upper bound for the \wep norm of the gradient $\grad \ell_{\alpha}^{\frac{1}{2}}(X)$, where $\alpha$ realizes the systole of $X$ (see Proposition \ref{gl-qi} in Section \ref{sys-wg}). The step for the thick part almost takes up the entirety of Section \ref{sys-wg}. Then, Theorem \ref{mt-lip} follows by integrating along the \wep geodesic segment and the Cauchy-Schwartz inequality. See Section \ref{sys-wg} for more details.

For any $\epsilon>0$, let $\sM^{\geq \epsilon}_{g,n}$ be the $\epsilon$-thick part of moduli space. The Mumford compactness theorem tells that $\sM^{\geq \epsilon}_{g,n}$ is compact. Denote by $\partial \sM^{\geq \epsilon}_{g,n}$ the boundary of $\sM^{\geq \epsilon}_{g,n}$, which consists of $\epsilon$-thick surfaces whose injectivity radii are $\epsilon$. It is clear that moduli space $\sM_{g,n}$ is foliated by $\partial \sM^{\geq \epsilon}_{g,n}$ for all $s>0$. The following result bounds the \wep distance between two leaves.
\begin{theorem}\label{leaf}
There exists a universal constant $K'>0$, independent of $g$ and $n$, such that for any $s>t\geq 0$, 
\[\frac{\sqrt{s}-\sqrt{t}}{K'} \leq \dist_{wp}(\partial \sM^{\geq s}_{g,n},\partial \sM^{\geq t}_{g,n})\leq K'(\sqrt{s}-\sqrt{t}).\]
\end{theorem}\

As stated above, the asymptotic behavior of the \wep volume of $\sM_{g,0}$ has been well studied as $g$ tends to infinity. We are grateful to Maryam Mirzakhani for bringing the following interesting question to our attention.
\begin{question}\label{q-mm}
Fix a constant $R>0$, are there any good upper bounds for the \wep volume $\Vol_{wp}(B(X;R))$ as $g$ tends to infinity? Here $B(X;R)=\{Y\in \Teich(S_{g,0}); \ \dist_{wp}(Y,X)<R\}$ is the \wep geodesic ball of radius $R$ centered at $X$.
\end{question}

The last part of this paper is to study Question \ref{q-mm}. Let $S_g=S_{g,0}$ be the closed surface of genus $g$ and $\Teich(S_g)$ be \tec space endowed with the \wep metric. Since the completion $\overline{\Teich(S_g)}$ of $\Teich(S_g)$ is not locally compact \cite{Wolpert03}, it is well-known that the \wep volume of a geodesic ball of finite radius  blows up if this ball in $\overline{\Teich(S_g)}$ contains a boundary point (see Proposition \ref{blow-up} for more details). Thus, we need to assume that the \wep geodesic balls in Question \ref{q-mm} stay away from the boundary of $\overline{\Teich(S_g)}$. For any positive constant $r_0$, we define
\begin{eqnarray*}
\sU(\Teich(S_{g}))^{\geq r_0}:=\{X_g \in \Teich(S_g); \dist_{wp}(X_g; \partial \overline{\Teich(S_g)})\geq r_0\}
\end{eqnarray*}
where $\partial \overline{\Teich(S_g)}$ is the boundary of $\Teich(S_g)$. The space $\sU(\Teich(S_{g}))^{\geq r_0}$ is the subset in $\Teich(S_g)$ which is at least $r_0$-distance to the boundary. By applying Theorem \ref{mt-g} and Teo's \cite{Teo09} uniform lower bound for the Ricci curvature on the thick part of $\Teich(S_g)$, we will show that the \wep volume of any \wep geodesic ball in $\sU(\Teich(S_{g}))^{\geq r_0}$ rapidly decays to $0$ as $g$ tends to infinity. More precisely,
\begin{theorem}\label{decay-0}
For any $r_0>0$, then for any constant $\epsilon>0$ we have
\[\sup_{B(X_g;r_g)\subset \sU(\Teich(S_{g}))^{\geq r_0}} \Vol_{wp}(B(X_g;r_g))=o((\frac{1}{g})^{(3-\epsilon)g})\]
where the supremum is taken over all the geodesic balls in $\sU(\Teich(S_{g}))^{\geq r_0}$ and $B(X_g;r_g):=\{Y_g\in \Teich(S_g); \ \dist_{wp}(Y_g,X_g)<r_g\}$. 
\end{theorem}

\begin{remark}
 From Theorem \ref{mt-g} and Wolpert's upper bound for distance to strata (see Theorem \ref{d-stra}), the largest radius of \wep geodesic balls in $\sU(\Teich(S_{g}))^{\geq r_0}$ is comparable to $\sqrt{\ln{g}}$ as $g\to \infty$. In particular, Theorem \ref{decay-0} implies that for any constant $a \in (0,\frac{1}{2})$,
\[\lim\limits_{g\to \infty}  \displaystyle\inf_{X_g\in \Teich(S_g)}\vol_{wp}(B(X_g;(\ln{g})^{a}))=0.\]
\end{remark}

A direct consequence of Theorem \ref{decay-0} is the following result. 
\begin{corollary}\label{R-0}
Fix a constant $R>0$. Then there exists a constant $\epsilon(R)>0$, only depending on $R$, such that for any $\epsilon>0$,
\[\sup_{X_g\subset \sU(\Teich(S_{g}))^{\geq \epsilon(R)}} \Vol_{wp}(B(X_g;R))=o((\frac{1}{g})^{(3-\epsilon)g}).\]
In particular, $\displaystyle \lim_{g\to \infty}\displaystyle\sup_{X_g\subset \sU(\Teich(S_{g}))^{\geq \epsilon(R)}} \Vol(B(X_g;R))=0$.
\end{corollary}
The corollary above answers Question \ref{q-mm} at least following a certain
interpretation.\\

\noindent \textbf{Plan of the paper.} Section \ref{np} provides some necessary background and the basic properties on two-dimensional hyperbolic geometry and the Weil-Petersson metric.  In Section \ref{one-l} we will show that the systole function is piecewise real analytic along \wep geodesics, which will be applied to prove Theorem \ref{mt-lip}. We will prove  Theorem \ref{mt-lip} in Section \ref{sys-wg}. In Section \ref{g-n} we will prove Theorem \ref{leaf} and apply Theorem \ref{mt-lip} to prove Theorem \ref{mt-g} and \ref{mt-n}. In Section \ref{wp-v} we will establish Theorem \ref{decay-0} and Corollary \ref{R-0}. 


\section{Acknowledgements}
The author would like to thank Jeffrey Brock, Hugo Parlier and Michael Wolf for their interest and useful conversations. He also would like to thank Maryam Mirzakhani for helpful discussions concerning Section \ref{wp-v}. He especially would like to thank Scott Wolpert for invaluable discussions on the various aspects of this paper. Without these discussions, this paper would have been impossible to complete. Part of this work was completed while visiting the Chern Institute of Mathematics in June 2014, and while attending the special program entitled "Geometric Structures on 3-manifolds" at the Institute for Advanced Study in October 2015. The author would like to give thanks for their hospitality. Most of this work was finished when the author was a G. C. Evans Instructor at Rice University. He would like to thank the Department of Mathematics of Rice University for all of their support in the past several years.


\section{Notations and Preliminaries}\label{np} 
In this section we will set up the notations and provide some necessary background on two-dimensional hyperbolic geometry, \tec theory and the \wep metric.
\subsection{Hyperbolic upper half plane} Let $\mathbb{H}$ be the upper half plane endowed with the hyperbolic metric $\rho(z)|dz|^2$ where
$$\rho(z)=\frac{1}{(\im(z))^2}.$$

A geodesic line in $\mathbb{H}$ is either a vertical line or an upper semi-circle centered at some point on the real axis. For $z=(r,\theta) \in \mathbb{H}$ given in polar coordinate where $\theta \in (0,\pi)$, the hyperbolic distance between $z$ and the imaginary axis $\textbf{i}\mathbb{R^+}$ is
\begin{eqnarray}\label{i-dis}
\dist_{\mathbb{H}}(z, \textbf{i}\mathbb{R^+})=\ln|\csc{\theta}+|\cot{\theta}||.
\end{eqnarray}

Thus, 
\begin{eqnarray}\label{i-exp}
e^{-2\dist_{\mathbb{H}}(z, \textbf{i}\mathbb{R^+})}\leq \sin^2{\theta}=\frac{\Im^2(z)}{|z|^2}\leq 4e^{-2\dist_{\mathbb{H}}(z, \textbf{i}\mathbb{R^+})}.
\end{eqnarray}

It is known that any eigenfunction with positive eigenvalue of the hyperbolic Laplacian of $\mathbb{H}$ satisfies the mean value property \cite[Coro.1.3]{Fay77}. For $z=(r,\theta) \in \mathbb{H}$ given in polar coordinate, the function 
\[u(\theta)=1-\theta \cot{\theta}\] 
is a positive $2$-eigenfunction. Thus, $u$ satisfies the mean value property. It is not hard to see that $\min\{u(\theta), u(\pi-\theta)\}$ also satisfies the mean value property. Since $\min\{u(\theta), u(\pi-\theta)\}$ is comparable to $\sin^2{\theta}$, from inequality (\ref{i-exp}) we know that the function $e^{-2\dist_{\mathbb{H}}(z, \textbf{i}\mathbb{R^+})}$ satisfies the mean value property in $\mathbb{H}$. The following lemma is the simplest version of \cite[Lemma 2.4]{Wolpert08}.
\begin{lemma}\label{mvp}
For any $r>0$ and $p \in \mathbb{H}$, there exists a positive constant $c(r)$, only depending on $r$, such that
\begin{eqnarray*}
e^{-2\dist_{\mathbb{H}}(p, \textbf{i}\mathbb{R^+})}\leq c(r) \int_{B_{\mathbb{H}}(p;r)}e^{-2\dist_{\mathbb{H}}(z, \textbf{i}\mathbb{R^+})}dA(z)
\end{eqnarray*}
where $B_{\mathbb{H}}(p;r)=\{z\in \mathbb{H}; \dist_{\mathbb{H}}(p,z)< r\}$ is the hyperbolic geodesic ball of radius $r$ centered at $p$ and $dA(z)$ is the hyperbolic area element. 
\end{lemma}

\subsection{\tec space}
Let $S_{g,n}$ be a surface of genus $g$ with $n$ punctures which satisfies that $3g-3+n>0$. Let $\textsl{M}_{-1}$ be the space of Riemannian metrics on $S_{g,n}$ with constant curvatures $-1$, and $X=(S_{g,n},\sigma|dz|^2) \in \textsl{M}_{-1}$. The group $\Diff_+$, which is the group of orientation-preserving diffeomorphisms, acts by pull back on $\textsl{M}_{-1}$. In particular this holds for the normal subgroup $\Diff_0$, the group of diffeomorphisms 
isotopic to the identity. The group $\Mod(S_{g,n}):=\Diff_+/\Diff_0$ is called the \textsl{mapping class group} of $S_{g,n}$. 

The \textsl{Teichm\"uller space} $\sT(S_{g,n})$ of $S_{g,n}$ is defined as
\begin{equation}
\nonumber \sT(S_{g,n}):=M_{-1}/\Diff_0.  
\end{equation}

The \textsl{moduli space} $\sM(S_{g,n})$ of $S_{g,n}$ is defined as
\begin{equation}
\nonumber \sM(S_{g,n}):=\sT(S_{g,n}) /\Mod(S_{g,n}).  
\end{equation}

The \tec space $\sT(S_{g,n})$ is a real analytic manifold. Let $\alpha$ be an essential simple closed curve on $S_{g,n}$, then for any $X \in \Teich(S_{g,n})$, there exists a unique closed geodesic $[\alpha]$ in $X$ which represents for $\alpha$ in the fundamental group of $S_{g,n}$. We denote by $\ell_{\alpha}(X)$ the length of $[\alpha]$ in $X$. In particular $\ell_{\alpha}(\cdot) $ defines a function on $\sT(S_{g,n})$. The following property is well-known.
\begin{lemma}\cite[Lemma 3.7]{IT92}\label{l-ana}
The geodesic length function $\ell_{\alpha}(\cdot):\sT(S_{g,n}) \to \mathbb{R}^+$ is real-analytic.
\end{lemma} 

Let $X\in \sT(S_{g,n})$ be a hyperbolic surface. The $\textsl{systole}$ of $X$ is the length of a shortest essential simple closed geodesic in $X$. We denote by $\ell_{sys}(X)$ the systole of $X$. It defines a continuous function $\lsys(\cdot): \sT(S_{g,n})\to \mathbb{R}^+$, which is called the \textsl{systole function}. In general, the systole function is clearly continuous and not smooth because of corners where there may exist multiple essential simple closed geodesics realizing the systole. This function is very useful in \tec theory. Curves that realize the systole are often referred to \emph{systolic curves}. One may refer to \cite{Akr03, Gen15, Sch93} for more details. In this paper we will study the behavior of this function along \wep geodesics and apply these results to different problems.

Fixed a constant $\epsilon_0>0$. The \textsl{$\epsilon_0$-thick part} of \tec space of $S_{g,n}$, denoted by $\sT(S_{g,n})^{\geq \epsilon_0}$, is defined as follows.
\[\sT(S_{g,n})^{\geq \epsilon_0}:=\{X\in \sT(S_{g,n}); \ \lsys(X)\geq \epsilon_0\}.\]
The space $\sT(S_{g,n})^{\geq \epsilon_0}$ is invariant by the mapping class group. The \textsl{$\epsilon_0$-thick part} of moduli space of $S_{g,n}$, denoted by $\sM(S_{g,n})^{\geq \epsilon_0}$, is defined by
\[\sM(S_{g,n})^{\geq \epsilon_0}:=\sT(S_{g,n})^{\geq \epsilon_0}/\Mod(S_{g,n}).\]

It is known that $\sM(S_{g,n})^{\geq \epsilon_0}$ is compact for all $\epsilon_0>0$, which is due to Mumford \cite{Mumford71}. For more details on \tec theory, one may refer to \cite{IT92, Hubbard06}.

\subsection{\wep metric}
The real-analytic space $\sT(S_{g,n})$ carries a natural complex structure.  Let $X=(S_{g,n},\sigma(z)|dz|^2) \in \sT_{g,n}$ be a point. The tangent space at $X$ is identified with the space of harmonic Beltrami differentials on $X$ which are forms of 
$\mu=\frac{\overline{\psi}}{\sigma}$ where $\psi$ is a holomorphic quadratic differential on $X$. Let $dA(z)=\sigma(z) dxdy$ be the volume form of $X=(S_{g,n},\sigma(z)|dz|^2)$ where $z=x+y \textbf{i}$. The \textit{Weil-Petersson metric} is the Hermitian
metric on $\sT(S_{g,n})$ arising from the the \textit{Petersson scalar  product}
\begin{equation}
 \langle \varphi,\psi \rangle_{WP}= \int_{X} \frac{\varphi (z)}{\sigma(z)}  \frac{\overline{\psi(z)}}{\sigma(z)} dA(z)\nonumber
\end{equation}
via duality. We will concern ourselves primarily with its Riemannian part $g_{WP}$. We denote by $\Teich(S_{g,n})$ the Teichm\"uller space endowed with the Weil-Petersson metric. The mapping class group $\Mod(S_{g,n})$ acts properly discontinuously on $\Teich(S_{g,n})$ by isometries. Reversely, from Masur-Wolf \cite{MW02} and Brock-Margalit \cite{BM07}  the whole isometry group of $\Teich(S_{g,n})$ is exactly the extended mapping class group except for some low complexity cases. The \wep metric on \tec space descends into a metric on moduli space. We denote by $\sM_{g,n}$ moduli space $\sM(S_{g,n})$ endowed with the \wep metric. 

The space $\Teich(S_{g,n})$ is incomplete \cite{Chu76, Wolpert75}, negatively curved \cite{Tromba86, Wolpert86} and uniquely geodesically convex \cite{Wolpert87}. The moduli space $\sM_{g,n}$ is an orbifold with finite volume and finite diameter. One may refer to \cite{IT92, Wolpertbook} for more details on the \wep metric. The following fundamental fact is due to Ahlfors \cite{Ahlfors61}, which will be used later.
\begin{theorem}[Ahlfors]\label{ahlf}
The space $\Teich(S_{g,n})$ is real-analytic K\"ahler.
\end{theorem}

The following convexity theorem is due to Wolpert \cite{Wolpert87}. He used this result to give a new solution to the Nielsen Realization Problem which was first solved by Kerckhoff \cite{Ker83}. An alternative proof of this convexity theorem was given by Wolf \cite{Wolf12}, through using harmonic map theory.
\begin{theorem}[Wolpert]\label{conv}
For any essential simple closed curve $\alpha\subset S_{g,n}$, the length function $\ell_{\alpha}:\Teich(S_{g,n})\to \mathbb{R}^+$ is strictly convex.
\end{theorem}

\subsection{Augmented Teichm\"uller space}
The non-completeness of the Weil-Petersson metric corresponds to finite-length geodesics in $\Teich(S_{g,n})$ along which some essential simple closed curve pinches to zero. In \cite{Masur76} the completion $\overline{\Teich(S_{g,n})}$ of $\Teich(S_{g,n})$, called the \textsl{augmented Teichm\"uller space}, is described concretely by adding strata consisting of stratum $\sT_{\sigma}$ defined by the vanishing of lengths
\begin{eqnarray*}
\ell_{\alpha}=0
\end{eqnarray*}
for each $\alpha \in \sigma$ where $\sigma$ is a collection of mutually disjoint essential simple closed curves. The stratum $\sT_{\sigma}$ are naturally products of lower dimensional Teichm\"uller spaces corresponding to the nodal surfaces in $\sT_{\sigma}$ \cite{Masur76}. The space $\overline{\Teich(S_{g,n})}$ is a complete CAT(0) space. It was shown in \cite{DW03, Wolpert03, Yamada04} that every stratum $\sT_{\sigma}$ is totally geodesic in $\overline{\Teich(S_{g,n})}$. Since the completion $\overline{\sT_{\sigma}}$ of $\sT_{\sigma}$ is convex in $\overline{\Teich(S_{g,n})}$, by elementary CAT(0) geometry (see\cite{BH}) the nearest projection map 
\[\pi_{\sigma}:\Teich(S_{g,n}) \to \overline{\sT_{\sigma}}\]
is well-defined. Using Wolpert's theorem on the structure of the Alexandrov tangent cone at the boundary of $\overline{\Teich(S_{g,n})}$ (see Theorem 4.18 in \cite{Wolpert08}) and the first variation formula for the distance function, one can show that for any $X\in \Teich(S_{g,n})$, the image $\pi_{\sigma}(X)$ is contained in $\sT_{\sigma}$. One can see more details in \cite{Fuj2013, Wu12}. 

The following result of Wolpert (see section 4 in \cite{Wolpert08} for more details) will be used to prove the upper bounds in Theorems \ref{mt-g} and \ref{mt-n}. Denote by $\dist_{wp}(\cdot, \cdot)$ the \wep distance.  

\begin{theorem}[Wolpert]\label{d-stra}
For any $X\in \Teich(S_{g,n})$, then we have 
\[\dist_{wp}(X,\pi_{\sigma}(X))\leq \sqrt{2\pi \cdot \sum_{\alpha \in \sigma^0}\ell_{\alpha}(X)}.\]
\end{theorem}

It was shown by Masur \cite{Masur76} that the completion $\overline{\sM}_{g,n}$ of moduli space $\sM_{g,n}$ is homeomorphic to the Deligne-Mumford compactification of moduli space. Recall that the inradius 
$\InR(\sM_{g,n})$ of $\sM_{g,n}$ is defined as $\max_{X\in \sM_{g,n}}\dist_{wp}(X, \partial \overline{\sM}_{g,n})$.
The inradius $\InR(\sM_{g,n})$ is the largest radius of geodesic balls in the interior of $\overline{\sM}_{g,n}$. Similarly, we also define the \textsl{inradius} $\InR(\Teich(S_{g,n}))$ of $\Teich(S_{g,n})$ as 
\[\InR(\Teich(S_{g,n})):=\max_{X\in \Teich(S_{g,n})}\dist_{wp}(X, \partial \overline{\Teich(S_{g,n})})\]
where $\partial \overline{\Teich(S_{g,n})}$ is the boundary of $\overline{\Teich(S_{g,n})} $.

In this article we will study the asymptotic behaviors of $\InR(\sM_{g,n})$ and $\InR(\Teich(S_{g,n}))$ either as $g$ goes to infinity or as $n$ goes to infinity.


\section{The systole function is piecewise real analytic}\label{one-l}
As stated in Section \ref{np}, although the systole function $\lsys(\cdot)$ is continuous over $\Teich(S_{g,n})$, it is not smooth. In this section we will provide two fundamental lemmas on the systole function $\lsys(\cdot)$ along a \wep geodesic such that we can take the derivative of the systole function along the \wep geodesic, which are crucial in the proof of Theorem \ref{mt-lip}. Before stating the results, we provides three basic claims on geodesic length functions. We always assume \wep geodesics use arc-length parameters.

\begin{claim}\label{f-1}
For any essential simple closed curve $\alpha \subset S_{g,n}$ and $\gamma:[0,s]\to \Teich(S_{g,n})$ be a \wep geodesic where $s>0$ is a constant. Then the geodesic length function $\ell_{\alpha}(\gamma(t)):[0,s]\to \mathbb{R}^+$ is real-analytic on $t$.
\end{claim}

\begin{proof}[Proof of Claim \ref{f-1}]
From Lemma \ref{ahlf} we know that $\Teich(S_{g,n})$ is real-analytic. In particular, all the Christoffel symbols are real-analytic. Thus, the classical Cauchy-Kowalevski Theorem gives that the solution of the \wep geodesic equation is real-analytic. That is, every \wep geodesic is real-analytic. Then the claim follows from Lemma \ref{l-ana}.
\end{proof}

Let $X\in \Teich(S_{g,n})$. We define the set $\ssys(X)$ of systolic curves as 
\[\ssys(X):=\{\beta \subset S_{g,n};\  \ell_{\beta}(X)=\lsys(X)\}.\]
It is clear that the set $\ssys(X)$ is finite for all $X\in \Teich(S_{g,n})$.

\begin{claim}\label{f-2}
Let $s>0$ and $\gamma:[0,s]\to \Teich(S_{g,n})$ be a \wep geodesic. Then the union $\displaystyle \cup_{0\leq t\leq s} \ssys(\gamma(t))$ is a finite set. 
\end{claim}
\begin{proof}[Proof of Claim \ref{f-2}]
First we denote by $\dist_{T}(\cdot, \cdot)$ the \tec distance. Since the image $\gamma([0,s])$ is a compact subset in $\Teich(S_{g,n})$, there exists a constant $K>0$ such that the \tec distance $$\max_{t\in [0,s]}\dist_{T}(\gamma(0), \gamma(t))\leq K$$ and $$\max_{t\in [0,s]}\lsys(\gamma(t))\leq K.$$  
By \cite[Lemma 3.1]{Wolpert79} we know that for all $t\in [0,s]$ and $\beta(t) \in \ssys(\gamma(t))$ we have $\ell_{\beta(t)}(\gamma(0))\leq K\cdot e^{2K}$. That is, the union satisfies $$\displaystyle \cup_{0\leq t\leq s} \ssys(\gamma(t))\subset \{\beta \subset S_{g,n}; \ell_{\beta}(\gamma(0))\leq K\cdot e^{2K}\}$$ which is a finite set. Then the claim follows.
\end{proof}

\noindent We do not know whether the cardinality of the union $\displaystyle \cup_{0\leq t\leq s} \ssys(\gamma(t))$ in the lemma above has any precise upper bound.

\begin{claim}\label{f-3}
Let $s>0$ be a constant, the curve $\gamma:[0,s]\to \Teich(S_{g,n})$ be a \wep geodesic and $\alpha, \beta \in \ssys(\gamma(0))$ be two distinct essential simple closed geodesics. Then either $\ell_{\alpha}(\gamma(t))\equiv\ell_{\beta}(\gamma(t))$ over $[0,s]$ or there exists a constant $0<s_0\leq s$ such that either $\ell_{\alpha}(\gamma(t))<\ell_{\beta}(\gamma(t))$ over $(0,s_0)$ or $\ell_{\beta}(\gamma(t))<\ell_{\alpha}(\gamma(t))$ over $(0,s_0)$.
\end{claim}
\begin{proof}[Proof of Claim \ref{f-3}]
Since the image $\gamma([0,s])$ is contained in $\Teich(S_{g,n})$, we can extend the geodesic $\gamma([0,s])$ in both directions a little bit longer. That is, there exists a positive constant $\epsilon>0$ such that $\gamma: (-\epsilon,s+\epsilon)\to \Teich(S_{g,n})$ is well-defined. By Claim \ref{f-1} we know that both $\ell_{\alpha}$ and $\ell_{\beta}$ are real-analytic along the \wep geodesic $\gamma(-\epsilon,s+\epsilon)$. If all the derivatives $\ell_{\alpha}^{(k)}(\gamma(0))=\ell_{\beta}^{(k)}(\gamma(0))$ for all $k \in \mathbb{N}^+$, then the Taylor expansions of $\ell_{\alpha}$ and $\ell_{\beta}$ at $\gamma(0)$ tells that $\ell_{\alpha}(\gamma(t))\equiv\ell_{\beta}(\gamma(t))$ over $[0,s]$. Otherwise, there exists a positive integer $k_0$ such that $\ell_{\alpha}^{(k)}(\gamma(0))=\ell_{\beta}^{(k)}(\gamma(0))$ for all $0\leq k \leq k_0-1$ and $\ell_{\alpha}^{(k_0)}(\gamma(0))\neq \ell_{\beta}^{(k_0)}(\gamma(0))$. The Taylor expansions of $\ell_{\alpha}$ and $\ell_{\beta}$ at $\gamma(0)$ clearly imply the later case of the claim.
\end{proof}

Now we are ready to state the first lemma, which will be applied to prove Proposition \ref{lip-sys}.
\begin{lemma}\label{key-l}
Let $X\neq Y\in \Teich(S_{g,n})$, $s=\dist_{wp}(X,Y)>0$ and $\gamma:[0,s]\to \Teich(S_{g,n})$ be the \wep geodesic with $\gamma(0)=X$ and $\gamma(s)=Y$. Then there exist a positive integer $k$, a partition $0=t_0<t_1<\cdots<t_{k-1}<t_k=s$ of the interval $[0,s]$ and a sequence of essential simple closed curves $\{\alpha_{i}\}_{0\leq i \leq k-1}$ in $S_{g,n}$ such that for all $0\leq i \leq k-1$,\\
$(1).  \ \alpha_i \neq \alpha_{i+1}.$\\
$(2). \ \ell_{\alpha_i}(\gamma(t))=\lsys(\gamma(t)), \quad \forall t_i\leq t \leq t_{i+1}.$
\end{lemma}

\begin{proof}
First by Claim \ref{f-2} one may assume that the union $$\displaystyle \cup_{0\leq t\leq s} \ssys(\gamma(t))=\{\beta_i\}_{1\leq i \leq n'}$$ for some positive integer $n'$ where $\beta_i \subset S_{g,n}$ is an essential simple closed curve for each $1\leq i \leq n'$. Without loss of generality one may assume that $\ssys(\gamma(0))$ consists of the first $n_0$ curves for some $0<n_0\leq n'$. That is
$$\ssys(\gamma(0))=\cup_{1\leq i \leq n_0}\{\beta_{i}\}.$$ 

Thus, for all $1\leq i \leq n_0$ and $n_0+1 \leq j \leq n'$ we have
\[\ell_{\beta{i}}(\gamma(0))<\ell_{\beta{j}}(\gamma(0)).\]

By the inequality above and using Claim \ref{f-3} finite number of steps (induction on $n_0$), there exist a positive constant $s_0\leq s$ and an essential simple closed curve in the set of systolic curves $\ssys(\gamma(0))$ of $\gamma(0)$, which is denoted by $\alpha_0$, such that for all $1\leq i \leq n'$ we have
\[\ell_{\alpha_0}(\gamma(t))\leq \ell_{\beta_i}(\gamma(t)), \ \forall \ 0\leq t \leq s_0.\]

Set 
\[t_1=\max\{t';\  \ell_{\alpha_0}(\gamma(t))\leq  \displaystyle\min_{1\leq i \leq n'}\ell_{\beta_i}(\gamma(t)), \ \forall 0\leq t\leq t'\}.\] 

In particular, 
\[  \ell_{\alpha_0}(\gamma(t))=\lsys(\gamma(t)), \ \forall 0\leq t\leq t_1.\] 

It is clear that 
\[0<s_0\leq t_1\leq s.\]

We may assume that $t_1<s$; otherwise we are done. 

Using the same argument above at $\gamma(t_1)$ there exist a positive constant $t_2$ with $t_1<t_2\leq s$ and an essential simple closed curve in $\ssys(\gamma(t_1))$, which is denoted by $\alpha_1$, such that
\[  \ell_{\alpha_1}(\gamma(t))\leq  \displaystyle\min_{1\leq i \leq n'}\ell_{\beta_i}(\gamma(t)), \ \forall t_1\leq t\leq t_2.\] 

In particular, 
\[  \ell_{\alpha_1}(\gamma(t))=\lsys(\gamma(t)), \ \forall t_1\leq t\leq t_2.\] 

From the definition of $t_1$ we know that 
\[\alpha_0 \neq \alpha_1.\]

Thus, from Claim \ref{f-3} and the definition of $t_1$ we know that there exists a constant $r_1>0$ with $r_1<t_2-t_1$ such that 
\[  \ell_{\alpha_1}(\gamma(t))< \ell_{\alpha_0}(\gamma(t)), \ \forall t_1< t< t_1+r_1.\]\

Then the conclusion follows by a finite induction. 

We argue by contradiction. If not, then there exist two infinite sequences of positive constants $\{t_i\}_{i\geq 1}$ with $t_i<t_{i+1}<s$,   $\{r_i\}_{i\geq 1}$ with $0<r_i<t_{1+i}-t_i$, and a sequence of essential simple closed curves $$\{\alpha_{i}\}_{i\geq 1}\subset \displaystyle \cup_{0\leq t\leq s} \ssys(\gamma(t))=\{\beta_i\}_{1\leq i \leq n'}$$ such that for all $i\geq 1$,
\begin{eqnarray}
 \ell_{\alpha_i}(\gamma(t))&=&\lsys(\gamma(t)), \ \forall t_i\leq t\leq t_{i+1}.\label{1-1-1}\\
\alpha_i &\neq& \alpha_{i-1}.\\
\ell_{\alpha_{i}}(\gamma(t))&<& \ell_{\alpha_{i-1}}(\gamma(t)), \ \forall t_i< t< t_i+r_i. \label{1-1-3}
\end{eqnarray}

Since $\{t_i\}$ is a bounded increasing sequence, we assume that $\displaystyle \lim_{i\to \infty}t_i=T$. It is clear that $0<T\leq s$. Since $\{\alpha_{i}\}_{i\geq 1}\subset \displaystyle \cup_{0\leq t\leq s} \ssys(\gamma(t))=\{\beta_i\}_{1\leq i \leq n'}$ which is a finite set, there exist two essential simple closed curves $\alpha \neq \beta \in \{\beta_i\}_{1\leq i \leq n'}$, a subsequence $\{t'_{i}\}_{i\geq 1}$ of $\{t_{2i}\}_{i\geq 1}$ and a subsequence $\{t''_{i}\}_{i\geq 1}$ of $\{t_{2i}+\frac{r_{2i}}{2}\}_{i\geq 1}$ such that for all $i\geq 1$,
\begin{eqnarray}
t'_i&<&t''_i<t'_{i+1}.\\
\lim_{i\to \infty}t'_i &=&\lim_{i\to \infty}t''_i= T. \label{1-1-4-0}\\
\ell_{\alpha}(\gamma(t'_i))&=&\lsys(\gamma(t'_i)). \label{1-1-4}\\
\ell_{\beta}(\gamma(t''_i))&=&\lsys(\gamma(t''_i)). 
\end{eqnarray}

Recall that $t''_i$ is of form $t_{2i}+\frac{r_{2i}}{2}$, Equation (\ref{1-1-3}) tells us that
\begin{eqnarray} \label{1-1-6}
\ell_{\beta}(\gamma(t''_i))&=&\lsys(\gamma(t''_i))<\ell_{\alpha}(\gamma(t''_i)). 
\end{eqnarray}

Since geodesic length functions are continuous over $\Teich(S_{g,n})$, 
\[ \ell_{\alpha}(\gamma(T))=\ell_{\beta}(\gamma(T))=\lsys(\gamma(T)).\]
Consider the \wep geodesic $c:[0,T]\to \Teich(S_{g,n})$ which is defined as $c(t)=\gamma(T-t)$ for all $0\leq t \leq T$. We apply Claim \ref{f-3} to $c$ at $c(0)=\gamma(T)$. Then from inequality (\ref{1-1-6}) and Claim \ref{f-3} we know that there exists a constant $s'_0>0$ such that 
\begin{eqnarray}\label{1-1-e}
\ell_{\beta}(c(t))<\ell_{\alpha}(c(t)), \ \forall t \in (0, s'_0).
\end{eqnarray}

On the other hand, from Equations (\ref{1-1-4-0}) and (\ref{1-1-4}) one may choose a number $\epsilon\in (0,s_0')$ to be small enough such that 
\begin{eqnarray}
\ell_{\alpha}(c(\epsilon))=\ell_{\alpha}(\gamma(T-\epsilon))=\lsys(\gamma(T-\epsilon))=\lsys(c(\epsilon))
\end{eqnarray}
which contradicts inequality (\ref{1-1-e}).
\end{proof}

For any $\epsilon_0>0$ we denote by $\Teich(S_{g,n})^{\geq \epsilon_0}$ the $\epsilon_0$-thick part of \tec space endowed with the \wep metric. Let $\Teich(S_{g,n})^{> \epsilon_0}$ be the interior of $\Teich(S_{g,n})^{\geq \epsilon_0}$. The following lemma will be applied to prove Theorem \ref{mt-lip}.
\begin{lemma}\label{key-l-2}
Fix a constant $\epsilon_0>0$. Let $X\neq Y\in \Teich(S_{g,n})$, $s=\dist_{wp}(X,Y)>0$ and $\gamma:[0,s]\to \Teich(S_{g,n})$ be the \wep geodesic with $\gamma(0)=X$ and $\gamma(s)=Y$. Then there exist a positive integer $k$, a partition $0=t_0<t_1<\cdots<t_{k-1}<t_k=s$ of the interval $[0,s]$, a sequence of closed intervals $\{[a_i,b_i]\subseteq [t_{i},t_{i+1}]\}_{0\leq i \leq k-1}$ and a sequence of essential simple closed curves $\{\alpha_{i}\}_{0\leq i \leq k-1}$ in $S_{g,n}$ such that for all $0\leq i \leq k-1$,\\
$(1). \ \alpha_i \neq \alpha_{i+1}.$\\
$(2). \ \ell_{\alpha_i}(\gamma(t))=\lsys(\gamma(t)), \quad \forall t_i\leq t \leq t_{i+1}.$\\
$(3). \ \gamma([0,s])\cap (\Teich(S_{g,n})-\Teich(S_{g,n})^{> \epsilon_0})=\displaystyle\cup_{0\leq i \leq k-1}\gamma([a_i,b_i]).$ 
\end{lemma}
\begin{proof}
First we apply Lemma \ref{key-l} to the \wep geodesic $\gamma([0,s])$. Then there exist a positive integer $k$, a partition $0=t_0<t_1<\cdots<t_{k-1}<t_k=s$ of the interval $[0,s]$ and a sequence of essential simple closed curves $\{\alpha_{i}\}_{0\leq i \leq k-1}$ in $S_{g,n}$ such that for all $0\leq i \leq k-1$ we have
\begin{eqnarray}\label{l-2-1}
\ell_{\alpha_i}(\gamma(t))=\lsys(\gamma(t)), \quad \forall t_i\leq t \leq t_{i+1}.
\end{eqnarray}

Thus, Part (1) and (2) follows. 

We apply Theorem \ref{conv} to the geodesic length function $$\ell_{\alpha_i}(\cdot): \gamma([t_i,t_{i+1}])\to \mathbb{R}^{+}$$ for all $0\leq i \leq k-1$. Since $\ell_{\alpha_i}(\cdot)$ is strictly convex on $\gamma([t_i,t_{i+1}])$ and $\gamma([0,s])\subset \Teich(S_{g,n})$, the maximal principle for a convex function gives that $\ell_{\alpha_i}^{-1}([0,\epsilon_0])$ is a closed connected subset in $\gamma([t_i,t_{i+1}])$, which is denoted by $\gamma([a_i,b_i])$ for some closed interval $[a_i,b_i] \subseteq [t_i,t_{i+1}]$ (note that $\gamma([a_i,b_i])$ may be just a single point or an empty set). Then Part (3) clearly follows from the choices of $a_i$ and $b_i$.
\end{proof}

\section{Uniformly Lipschitz}\label{sys-wg}
Recall that the systole function $\lsys(\cdot):\Teich(S_{g,n})\to \mathbb{R}^+$ is continuous and not smooth. The goal of this section is to prove Theorem \ref{mt-lip} which says that the square root of the systole function is uniformly Lipschitz continuous along \wep geodesics. The method in this section is influenced by \cite{Wolpert08}. For convenience we restate Theorem \ref{mt-lip} here.
  \begin{thm}
There exists a universal constant $K>0$, independent of $g$ and $n$, such that for all $X,Y\in \Teich(S_{g,n})$,
\[|\sqrt{\ell_{sys}(X)}-\sqrt{\ell_{sys}(Y)}|\leq K\dist_{wp}(X,Y).\]
  \end{thm}

We begin by outlining the idea of the proof.

For any \wep geodesic $\g(X,Y)\subset \Teich(S_{g,n})$ joining $X$ and $Y$ in $\Teich(S_{g,n})$, first we apply Lemma \ref{key-l-2} to make a thick-thin decomposition for the geodesic $\g(X,Y)$ such that both of the thick and thin parts are disjoint closed intervals with certain properties. Then we use different arguments for these two parts. For the thin part we will apply the following result due to Wolpert.
\begin{lemma}\cite[Lemma 3.16]{Wolpert08}\label{short}
There exists a universal constant $c>0$, independent of $g$ and $n$, such that for all $X\in \Teich(S_{g,n})$ and any essential simple closed curve $\alpha \subset S_{g,n}$,
\[\langle \grad \ell_{\alpha}, \grad \ell_{\alpha}\rangle_{wp}(X)\leq c \cdot (\ell_{\alpha}(X)+\ell_{\alpha}^2(X)e^{\frac{\ell_{\alpha}(X)}{2}}).\]
\end{lemma}

Fix a constant $k_0>0$, the lemma above implies that for all essential simple closed curve $\alpha \subset S_{g,n}$ with $\ell_{\alpha}\leq k_0$,
\[\langle \grad \ell_{\alpha}, \grad \ell_{\alpha}\rangle_{wp}(X)\leq C(k_0)\ell_{\alpha}\]
where $C(k_0)$ is a constant only depending on $k_0$.

Recall that the length $\ell_{\alpha}$ could be arbitrarily large for any essential simple closed curve $\alpha \subset S_{g,n}$ (Buser-Sarnak \cite{BS94} constructed hyperbolic surfaces whose injectivity radii grow roughly as $\ln{g}$), actually for the thick part of the geodesic $\g(X,Y)$, no matter how large the injectivity radius is, we will apply the following proposition, which is the main part of this section.
\begin{proposition}\label{lip-sys}
Fix a constant $\epsilon_0>0$. Then there exists a positive constant $C(\epsilon_0)$, only depending on $\epsilon_0$, such that for any $X, Y \in \Te(S_{g,n})$ with the \wep geodesic $\g(X,Y)\subset \Te(S_{g,n})^{\geq \epsilon_0}$, we have 
\[|\sqrt{\ell_{sys}(X)}-\sqrt{\ell_{sys}(Y)}|\leq C(\epsilon_0) \dist_{wp}(X,Y).\]
\end{proposition}\

For any essential simple closed curve $\alpha\subset S_{g,n}$, the geodesic length function $\ell_{\alpha}(\cdot)$ is real-analytic over $\Teich(S_{g,n})$. Gardiner in \cite{Gar75, Gar86} provided formulas for the differentials of $\ell_{\alpha}$. Let $(X, \sigma(z)|dz|^2) \in \Teich(S_{g,n})$ be a hyperbolic surface and $\Gamma$ be its associated Fuchsian group. Since $\alpha$ is an essential simple closed curve, we may denote by $A$ be the deck transformation on the upper half plane $\mathbb{H}$ corresponding to the simple closed geodesic $[\alpha]\subset X$. Consider the quadratic differential
\begin{equation}
\Theta_{\alpha}(z)=\displaystyle \sum_{E\in <A>/\Gamma} \frac{E'(z)^2}{E(z)^2}dz^2
\end{equation}
where $<A>$ is the cyclic group generated by $A$. 

Then the gradient $\grad \ell_{\alpha}(\cdot)$ of the geodesic length function $\ell_{\alpha}$ is 
\[\grad \ell_{\alpha}(X)(z)=\frac{2}{\pi}\frac{\overline{\Theta}_{\alpha}(z)}{\rho(z)|dz|^2} \]
where $\rho(z)|dz|^2$ is the hyperbolic metric on the upper half plane. The tangent vector $t_{\alpha}=\frac{\textbf{i}}{2}\grad \ell_{\alpha}$ is the infinitesimal Fenchel-Nielsen right twist deformation \cite{Wolpert82}. 

In \cite{Rie05} Riera provided a formula for the \wep inner product of a pair of geodesic length gradients. Let $\alpha, \beta \subset X$ be two essential simple closed curves with $A, B\in \Gamma$ be its associated deck transformations with axes $\tilde{\alpha}, \tilde{\beta} $ on the upper half plane. Riera's formula \cite[Theorem 2]{Rie05} says that
\begin{eqnarray*} 
\langle\,\grad \ell_{\alpha},\grad \ell_{\beta}\rangle_{wp}(X)=\frac{2}{\pi}(\ell_{\alpha}\delta_{\alpha \beta}+\displaystyle \sum_{E\in <A>\backslash \Gamma/<B>} (u \ln{|\frac{u+1}{u-1}|}-2))
\end{eqnarray*}
for the Kronecker delta $\delta_{\cdot}$, where $u=u(\tilde{\alpha}, E\circ \tilde{\beta})$ is the cosine of the intersection angle if $\tilde{\alpha}$ and $E\circ \tilde{\beta}$ intersect and is otherwise $\cosh{(\dist_{\mathbb{H}}(\tilde{\alpha}, E\circ \tilde{\beta}))}$ where $\dist_{\mathbb{H}}(\tilde{\alpha}, E\circ \tilde{\beta})$ is the hyperbolic distance between the two geodesic lines. Riera's formula was applied in \cite{Wolpert08} to study \wep gradient of simple closed curves of short lengths. In this paper we will use Riera's formula to study the systolic curves which may have large lengths. 

In particular setting $\alpha=\beta$ in Riera's formula, then we have
\begin{equation}\label{Rie-f}
\langle\,\grad \ell_{\alpha},\grad \ell_{\alpha}\rangle_{wp}(X)=\frac{2}{\pi}(\ell_{\alpha}+\displaystyle \sum_{E\in \{<A>\backslash\Gamma/<A>-id\}} (u \ln{\frac{u+1}{u-1}}-2))
\end{equation}
where $u=\cosh{(\dist_{\mathbb{H}}(\tilde{\alpha}, E\circ \tilde{\alpha}))}$ and the double-coset of the identity element is omitted from the sum. We can view the formula above as a function on essential simple closed curves in $S_{g,n}$. In this section, we will evaluate this function at $\alpha\in \ssys(X)$ and make estimates to prove the following result, which is essential in the proof of Proposition \ref{lip-sys}.

\begin{proposition}\label{gl-qi}
Fix a constant $\epsilon_0>0$. Then there exists a positive constant $D(\epsilon_0)$, only depending on $\epsilon_0$, such that for any $X\in \Teich(S_{g,n})^{\geq \epsilon_0}$ and any systolic curve $\alpha \in \ssys (X)$ we have
\[\langle \grad \ell_{\alpha},\grad \ell_{\alpha}\rangle_{wp}(X)\leq D(\epsilon_0)\cdot \ell_{\alpha}(X).\]
\end{proposition}

\begin{remark}
From Riera's formula it is clear that 
\[\langle \grad \ell_{\alpha},\grad \ell_{\alpha}\rangle_{wp}(X)\geq \ell_{\alpha}(X).\]
Thus, $\langle \grad \ell_{\alpha},\grad \ell_{\alpha}\rangle_{wp}(X)$ is comparable to $ \ell_{\alpha}(X)$ under the same conditions as in Proposition \ref{gl-qi}.
\end{remark}

Before we prove Proposition \ref{gl-qi}, let's set up some notations and provide two lemmas. 

As stated above, we let $X\in \Teich(S_{g,n})$ be a hyperbolic surface and $\alpha \subset X$ be an essential simple closed curve. Up to conjugacy, we may assume that the closed geodesic $[\alpha]$ corresponds to the deck transformation $A: z\to e^{\ell_{\alpha}}\cdot z$ with axis $\tilde{\alpha}=\textbf{i}\mathbb{R}^+$ which is the imaginary axis and the fundamental domain $\mathbb{A}=\{z\in \mathbb{H}; 1\leq |z|\leq e^{\ell_{\alpha}}\}$. Let $\gamma_1, \gamma_2$ be two geodesic lines in $\mathbb{H}$. The distance $\dist_{\mathbb{H}}(\gamma_1,\gamma_2)$ is given by 
\[\dist_{\mathbb{H}}(\gamma_1,\gamma_2)=\inf_{p\in \gamma_1}\dist_{\mathbb{H}}(p, \gamma_2).\]

The following lemma says that any two lifts of the closed geodesic $[\alpha]$ in the upper half plane are uniformly separated. More precisely,
\begin{lemma}\label{3-l-1}
Fix a constant $\epsilon_0>0$. Then there exists a constant $C_{0}(\epsilon_0)>0$, only depending on $\epsilon_0$, such that for any $X\in \Teich(S_{g,n})^{\geq \epsilon_0}$, $\alpha \in \ssys(X)$ and all $B\in \{<A>\backslash\Gamma-id\}$ we have
\[\dist_{\mathbb{H}}(\tilde{\alpha},B\circ \tilde{\alpha})\geq \frac{\epsilon_0}{4}.\]
\end{lemma}

\begin{proof}
The proof follows from a standard argument in Riemannian geometry (the so-called closing lemma). Since $X\in \Teich(S_{g,n})^{\geq \epsilon_0}$ and $\alpha \in \ssys(X)$, for every point $m\in [\alpha]$, the closed geodesic in $X$ representing $\alpha$, we have the geodesic ball $B_{X}(m; \frac{\epsilon_0}{4})\subset X$, of radius $\frac{\epsilon_0}{4}$ centered at $m$, is isometric to a hyperbolic geodesic ball of radius $\frac{\epsilon_0}{4}$ in $\mathbb{H}$. Since $[\alpha]$ is a systolic curve and $X\in \Teich(S_{g,n})^{\geq \epsilon_0}$, the intersection $[\alpha]\cap B_{X}(m;\frac{\epsilon_0}{4})$ is a geodesic arc of length $\frac{\epsilon_0}{2}$ with the midpoint $m$.\\

Claim: $\dist_{\mathbb{H}}(\tilde{\alpha},B\circ \tilde{\alpha})\geq C_{0}(\epsilon_0)$ for all $B\in \{<A>\backslash\Gamma-id\}$.\\

We argue by contradiction for the proof of the claim. Suppose it does not hold. Then we let $p\in \tilde{\alpha}$ and $q \in B\circ \tilde{\alpha}$ such that 
\begin{eqnarray}\label{3-1}
\dist_{\mathbb{H}}(p,q)<\frac{\epsilon_0}{4}.
\end{eqnarray}

Let $B_{\mathbb{H}}(p;\frac{\epsilon_0}{4})\subset \mathbb{H}$ be the geodesic ball centered at $p$ of radius $\frac{\epsilon_0}{4}$. It is clear that the covering map 
\[\pi: B_{\mathbb{H}}(p;\frac{\epsilon_0}{4}) \to X\] 
is an isometric embedding. Thus,
\begin{eqnarray}\label{3-2}
\pi (B_{\mathbb{H}}(p;\frac{\epsilon_0}{4}) \cap \tilde{\alpha})=[\alpha]\cap B_{X}(\pi(p);\frac{\epsilon_0}{4})
\end{eqnarray}

Since the two geodesic lines $\tilde{\alpha}$ and $B\circ \tilde{\alpha}$ are disjoint, by inequality (\ref{3-1}) we know that $q \in B_{\mathbb{H}}(p;\frac{\epsilon_0}{4})-B_{\mathbb{H}}(p;\frac{\epsilon_0}{4}) \cap \tilde{\alpha}$. Since $q \in  B\circ \tilde{\alpha}$, 
\[\pi(q) \in [\alpha]\cap B_{X}(\pi(p);\frac{\epsilon_0}{4}) \]
which, together with Equation (\ref{3-2}), implies that the covering map $\pi: B_{\mathbb{H}}(p;\frac{\epsilon_0}{4}) \to X$ is not injective, which is a contradiction. 
\end{proof}

\begin{remark}
The condition $\alpha \in \ssys(X)$ is essential in Lemma \ref{3-l-1}. Otherwise, the estimate above may fail if one think about that case that the intersection of $[\alpha]$ with a geodesic ball of small radius is not connected.
\end{remark}

Recall that the axis $\tilde{\alpha}$ of the closed geodesic $[\alpha]\subset X$ in the upper half plane is the imaginary axis $\textbf{i}\mathbb{R}^{+}$. Let $B\in \{<A>\backslash\Gamma/<A>-id\}$. It is clear that the two geodesic lines $B\circ (\textbf{i}\mathbb{R}^{+})$ and $\textbf{i}\mathbb{R}^{+}$ are disjoint, and have disjoint boundary points at infinity. Since the distance function between two convex subsets in $\mathbb{H}$ is strictly convex (one may see \cite[Page 176]{BH} in a more general setting), there exists a unique point $p_B \in B\circ(\textbf{i}\mathbb{R}^{+})$ such that 
\[\dist_{\mathbb{H}}(p_B, \textbf{i}\mathbb{R}^{+})=\dist_{\mathbb{H}}(B\circ (\textbf{i}\mathbb{R}^{+}),\textbf{i}\mathbb{R}^{+}).\]

The goal of the following lemma is to study the position of the nearest projection point $p_B$ in $\mathbb{H}$.
\begin{lemma}\label{3-l-2}
Let $B\in \{<A>\backslash\Gamma/<A>-id\}$. Then there exists a representative $B' \in <A>\backslash\Gamma$ for $B$ such that 
\[1\leq r_{B'}\leq e^{\ell_{\alpha}}\]
where $p_{B'}=(r_{B'},\theta_{B'})$ in polar coordinate be the nearest projection point on $ B'\circ (\textbf{i}\mathbb{R}^{+})$ from $\textbf{i}\mathbb{R}^{+}$.
\end{lemma}
\begin{proof}
Recall that the fundamental domain of $A$, the deck transformation corresponding to $[\alpha]$, is $\mathbb{A}=\{z\in \mathbb{H}; 1\leq |z|\leq e^{\ell_{\alpha}}\}$. For any $B\in \{<A>\backslash\Gamma-id\}$, the map $B:\mathbb{A}\to \mathbb{A}$ is biholomorphic. Let $p_B=(r_B,\theta_B)$ in polar coordinates be the nearest projection point on $ B\circ (\textbf{i}\mathbb{R}^{+})$ from $\textbf{i}\mathbb{R}^{+}$.

Case (1): $1\leq r_B \leq e^{\ell_{\alpha}}$. Then we are done by choosing $B'=B$.

Case (2): $0<r_B<1$ or $r_B > e^{\ell_{\alpha}}$. First there exists an integer $k$ such that $$A^k \circ r_B \in \{(r, \theta)\in \mathbb{H};1\leq r\leq  e^{\ell_{\alpha}}\}.$$
Choose $B'=A^k\cdot B$. Then $B'=B \in \{<A>\backslash\Gamma/<A>-id\}$ by the definition of double-cosets. Since $B'=A^k\cdot B$ and $A^k$ acts on $\textbf{i}\mathbb{R}^{+}$ by isometries, 
\[\dist_{\mathbb{H}}(\textbf{i}\mathbb{R}^{+},B'\circ(\textbf{i}\mathbb{R}^{+}))=\dist_{\mathbb{H}}(\textbf{i}\mathbb{R}^{+},A^k \circ r_B).\]
Let $p_{B'}=(r_{B'},\theta_{B'})$ in polar coordinates be the nearest point projection on $ B'\circ (\textbf{i}\mathbb{R}^{+})$ from $\textbf{i}\mathbb{R}^{+}$. Then we have $1\leq r_{B'}\leq e^{\ell_{\alpha}}$.
\end{proof}

Recall that in Riera's formula (see Equation (\ref{Rie-f})) the function $(u \ln{\frac{u+1}{u-1}}-2)$ satisfies 
\[\lim_{u\to \infty} \frac{u \ln{\frac{u+1}{u-1}}-2}{u^{-2}}=\frac{2}{3}.\]
From Lemma \ref{3-l-2} we know that the quantity $u$ in Equation (\ref{Rie-f}) satisfies 
\[u\geq \cosh{(\frac{\epsilon_0}{4})}>1\]
provided that $X\in \Teich(S_{g,n})^{\geq \epsilon_0}$ and $\alpha \in \ssys(X)$. Thus, there exists a positive constant $C_{2}(\epsilon_0)$, depending  only on $\epsilon_0$, such that 
\begin{eqnarray}\label{u-u}
(u \ln{\frac{u+1}{u-1}}-2)\leq C_{2}(\epsilon_0) \cdot u^{-2}.
\end{eqnarray}

Now we are ready to prove Proposition \ref{gl-qi}.
\begin{proof}[Proof of Proposition \ref{gl-qi}]
We will apply Equation (\ref{Rie-f}) to finish the proof.

First from Equations (\ref{Rie-f}) and (\ref{u-u}) we have
\begin{eqnarray}
\langle \grad \ell_{\alpha},\grad \ell_{\alpha}\rangle_{wp}(X)&\leq& \frac{2}{\pi}(\ell_{\alpha}+C_{2}(\epsilon_0) \displaystyle \sum_{B\in \{<A>\backslash\Gamma/<A>-id\}} e^{-2\dist_{\mathbb{H}}(\textbf{i}\mathbb{R}^+,B\circ (\textbf{i}\mathbb{R}^+))}). \nonumber
\end{eqnarray}

Let $p_B \in B\circ (\textbf{i}\mathbb{R}^+)$ such that 
\[\dist_{\mathbb{H}}(p_B, \textbf{i}\mathbb{R}^{+})=\dist_{\mathbb{H}}(B\circ (\textbf{i}\mathbb{R}^{+}),\textbf{i}\mathbb{R}^{+}).\]

Then, 
\begin{small}\begin{eqnarray}\label{u-1}
\langle \grad \ell_{\alpha},\grad \ell_{\alpha}\rangle_{wp} \leq \frac{2}{\pi}(\ell_{\alpha}+C_{2}(\epsilon_0) \displaystyle \sum_{B\in \{<A>\backslash\Gamma/<A>-id\}} e^{-2\dist_{\mathbb{H}}(\textbf{i}\mathbb{R}^+,p_B)}).
\end{eqnarray}
\end{small}

Lemma \ref{mvp} implies that the function  $e^{-2\dist_{\mathbb{H}}(\textbf{i}\mathbb{R}^+,z)}$ has the mean value property. Set
\[r(\epsilon_0)=\frac{\epsilon_0}{8}.\]

Thus, from Lemma \ref{mvp} we know that
\begin{eqnarray*}
 e^{-2\dist_{\mathbb{H}}(\textbf{i}\mathbb{R}^+,p_B)}\leq  c(r(\epsilon_0)) \int_{B_{\mathbb{H}}(p_B;r(\epsilon_0))}e^{-2\dist_{\mathbb{H}}(z, \textbf{i}\mathbb{R^+})}dA(z)
\end{eqnarray*}
where $c(\cdot)$ is the constant in Lemma \ref{mvp}.

From our assumption that $X\in \Teich(S_{g,n})^{\geq \epsilon_0}$, Lemma \ref{3-l-1} and the triangle inequality we know that the geodesic balls $\{B_{\mathbb{H}}(p_B;r(\epsilon_0))\}_{B\in \{<A>\backslash\Gamma/<A>-id\}}$ are pairwise disjoint. Thus,

\begin{eqnarray*}
&& \sum_{B\in \{<A>\backslash\Gamma/<A>-id\}}  e^{-2\dist_{\mathbb{H}}(\textbf{i}\mathbb{R}^+,p_B)}\\
&\leq&  c(r(\epsilon_0))\sum_{B\in  \{<A>\backslash\Gamma/<A>-id\}} \int_{B_{\mathbb{H}}(p_B;r(\epsilon_0))}e^{-2\dist_{\mathbb{H}}(z, \textbf{i}\mathbb{R^+})}dA(z)\\ 
&=&c(r(\epsilon_0))\displaystyle\int_{\cup_{B\in \{<A>\backslash\Gamma/<A>-id\}}B_{\mathbb{H}}(p_B;r(\epsilon_0))}e^{-2\dist_{\mathbb{H}}(z, \textbf{i}\mathbb{R^+})}dA(z).
\end{eqnarray*}

Since $\ell_{\alpha}(X)\geq \epsilon_0$ and $r(\epsilon_0)\leq \frac{\epsilon_0}{4}$, from Lemma \ref{3-l-2} we have that the union of the geodesic balls satisfy that
\[\cup_{B\in \{<A>\backslash\Gamma/<A>-id\}}B_{\mathbb{H}}(p_B;r(\epsilon_0))\subset \{(r,\theta)\in \mathbb{H}; e^{-\ell_{\alpha}}\leq r\leq e^{2\ell_{\alpha}}\}.\]

Thus, 
\begin{eqnarray*}
 \sum_{B\in \{<A>\backslash\Gamma/<A>-id\}}  e^{-2\dist_{\mathbb{H}}(\textbf{i}\mathbb{R}^+,p_B)}\leq c(r(\epsilon_0))\times\\
 \displaystyle\int_{\{(r,\theta)\in \mathbb{H}; e^{-\ell_{\alpha}}\leq r\leq e^{2\ell_{\alpha}}\}}e^{-2\dist_{\mathbb{H}}(z, \textbf{i}\mathbb{R^+})}dA(z).
\end{eqnarray*}

From inequality (\ref{i-exp}) we have
\begin{eqnarray}\label{u-2}
 \sum_{B\in \{<A>\backslash\Gamma/<A>-id\}}  e^{-2\dist_{\mathbb{H}}(\textbf{i}\mathbb{R}^+,p_B)}&\leq&  c(r(\epsilon_0))\int_{0}^{\pi} \int_{e^{-\ell_{\alpha}}}^{e^{2\ell_{\alpha}}}\sin^2{\theta}dA(z)\\
 &=& c(r(\epsilon_0))\int_{0}^{\pi} \int_{e^{-\ell_{\alpha}}}^{e^{2\ell_{\alpha}}}\frac{\sin^2{\theta}}{r^2\sin^2{\theta}}rdrd\theta \nonumber\\
&=&c(r(\epsilon_0))\cdot 3\pi  \cdot\ell_{\alpha} \nonumber
\end{eqnarray}
where in the first equality we apply $dA(z)=\frac{|dz|^2}{y^2}=\frac{rdrd\theta}{r^2\sin^2\theta}$.

Therefore, the conclusion follows from inequalities (\ref{u-1}) and (\ref{u-2}) by choosing
\[D(\epsilon_0)=\frac{2}{\pi}(1+C_{2}(\epsilon_0)\cdot c(r(\epsilon_0))\cdot 3\pi).\]
\end{proof}

\begin{proof}[Proof of Proposition \ref{lip-sys}]
Let $s=\dist_{wp}(X,Y)>0$ and 
\[\gamma:[0,s]\to \Teich(S_{g,n})^{\geq \epsilon_0}\] 
be the geodesic $\g(X,Y)$ with $\gamma(0)=X$ and $\gamma(s)=Y$. From Lemma \ref{key-l} we know that there exist a positive integer $k$, a partition $0=t_0<t_1<\cdots<t_{k-1}<t_k=s$ of the interval $[0,s]$ and a sequence of essential simple closed curves $\{\alpha_{i}\}_{0\leq i \leq k-1}$ in $S_{g,n}$ such that for all $0\leq i \leq k-1$ we have
\[\ell_{\alpha_i}(\gamma(t))=\lsys(\gamma(t)), \quad \forall t_i\leq t \leq t_{i+1}.\]

Then, 
\begin{eqnarray*}
|\sqrt{\lsys(X)}-\sqrt{\lsys(Y)}|&\leq& \sum_{i=0}^{k-1}|\sqrt{\lsys(\gamma(t_i))}-\sqrt{\lsys(\gamma(t_{i+1}))}|\\
&=&  \sum_{i=0}^{k-1}|\sqrt{\ell_{\alpha_i}(\gamma(t_i))}-\sqrt{\ell_{\alpha_i}(\gamma(t_{i+1}))}|\\
&=& \sum_{i=0}^{k-1}|\int_{t_i}^{i+1} \langle \grad \ell_{\alpha_i}^{\frac{1}{2}}(\gamma(t)),\gamma'(t) \rangle_{wp}|dt\\
&\leq&  \sum_{i=0}^{k-1} \int_{t_i}^{t_{i+1}}||\grad \ell_{\alpha_i}^{\frac{1}{2}}(\gamma(t))||_{wp}dt
\end{eqnarray*}
where $||\cdot||_{wp}$ is the \wep norm.

Since $\gamma([0,s]) \subset \Teich(S_{g,n})^{\geq \epsilon_0}$, from Proposition \ref{gl-qi} we have for all $0\leq i \leq (k-1)$ and $t_i\leq t \leq t_{i+1}$,
\begin{eqnarray*}
||\grad \ell_{\alpha_i}^{\frac{1}{2}}(\gamma(t))||_{wp}\leq \sqrt{\frac{D(\epsilon_0)}{4}}.
\end{eqnarray*}

Recall that $\dist_{wp}(X,Y)=s=t_k$ and $t_0=0$. Therefore, the two inequalities above yield that
\begin{eqnarray*}
|\sqrt{\lsys(X)}-\sqrt{\lsys(Y)}|\leq \frac{\sqrt{D(\epsilon_0)}}{2}\dist_{wp}(X,Y).
\end{eqnarray*}

Then the conclusion follows by choosing $C(\epsilon_0)=\frac{\sqrt{D(\epsilon_0)}}{2}$.
\end{proof}

\begin{remark}
It is not hard to see that the constant $C(\epsilon_0)\to \infty $ as $\epsilon_0 \to 0$. 
\end{remark}

Before we prove Theorem \ref{mt-lip}, let us introduce the following result which is a direct consequence of Lemma \ref{short}.
\begin{lemma}\label{short-1}
There exists a universal constant $c>0$, independent of $g$ and $n$, such that for any $X\in \Teich(S_{g,n})$, and $\alpha \subset S_{g,n}$ which is an essential simple closed curve with $\ell_{\alpha}(X) \leq 1$, then the following holds
\[\langle \grad \ell_{\alpha}^{\frac{1}{2}}, \grad \ell_{\alpha}^{\frac{1}{2}}\rangle_{wp}(X)\leq c.\]
\end{lemma}

Now we are ready to prove Theorem \ref{mt-lip}.
\begin{proof}[Proof of Theorem \ref{mt-lip}]
Let $X\neq Y\in \Teich(S_{g,n})$, $s=\dist_{wp}(X,Y)>0$ and $\gamma:[0,s]\to \Teich(S_{g,n})$ be the \wep geodesic with $\gamma(0)=X$ and $\gamma(s)=Y$. We apply Lemma \ref{key-l-2} to the geodesic $\gamma([0,s])$ with $\epsilon_0=1$. So there exist a positive integer $k$, a partition $0=t_0<t_1<\cdots<t_{k-1}<t_k=s$ of the interval $[0,s]$, a sequence of closed intervals $\{[a_i,b_i]\subseteq [t_{i},t_{i+1}]\}_{0\leq i \leq k-1}$ and a sequence of essential simple closed curves $\{\alpha_{i}\}_{0\leq i \leq k-1}$ in $S_{g,n}$ such that for all $0\leq i \leq k-1$,
\begin{eqnarray}
\ell_{\alpha_i}(\gamma(t))&=&\lsys(\gamma(t)), \quad \forall t_i\leq t \leq t_{i+1}. \label{4-5-0}\\
\gamma([0,s])\cap \lsys^{-1}([0,1] )&=&\displaystyle\cup_{0\leq i \leq k-1}\gamma([a_i,b_i]). \label{4-5-1}
\end{eqnarray}

Since $[a_i,b_i]\subseteq [t_{i},t_{i+1]}$ for all $0\leq i \leq k-1$, from Equation (\ref{4-5-1}) we know that for all $0\leq i \leq k-1$, 
\begin{eqnarray}\label{4-5-3}
\lsys(\gamma(t))\geq 1, \ \forall t \in [t_i,a_i]\cup [b_i,t_{i+1}].
\end{eqnarray}

Then,
\begin{eqnarray*}
&&|\sqrt{\lsys(X)}-\sqrt{\lsys(Y)}|=|\sqrt{\lsys(\gamma(0))}-\sqrt{\lsys(\gamma(s))}|\\
&\leq& \sum_{i=0}^{k-1}(|\sqrt{\lsys(\gamma(t_i))}-\sqrt{\lsys(\gamma(a_i))}|+|\sqrt{\lsys(\gamma(a_i))}-\sqrt{\lsys(\gamma(b_i))}|\\
&+&|\sqrt{\lsys(\gamma(b_i))}-\sqrt{\lsys(\gamma(t_{i+1}))}|).
\end{eqnarray*}

From Equation (\ref{4-5-3}) and Proposition \ref{lip-sys} we have
\begin{eqnarray}\label{4-6-1}
&&|\sqrt{\lsys(X)}-\sqrt{\lsys(Y)}|\leq \sum_{i=0}^{k-1}(C(1)\cdot|a_i-t_i|+C(1)\cdot |t_{i+1}-b_i|)\\
&+&\sum_{i=0}^{k-1}|\sqrt{\lsys(\gamma(a_i))}-\sqrt{\lsys(\gamma(b_i))}| \nonumber\\
&=&\sum_{i=0}^{k-1}C(1)\cdot(t_{i+1}-t_i+a_i-b_i)+\sum_{i=0}^{k-1}|\sqrt{\lsys(\gamma(a_i))}-\sqrt{\lsys(\gamma(b_i))}|\nonumber\\
&=&\sum_{i=0}^{k-1}C(1)\cdot(t_{i+1}-t_i+a_i-b_i)+\sum_{i=0}^{k-1}|\sqrt{\ell_{\alpha_i}(\gamma(a_i))}-\sqrt{\ell_{\alpha_i}(\gamma(b_i))}|\nonumber
\end{eqnarray}
where we apply Equation (\ref{4-5-0}) in the last step.

Using Equations (\ref{4-5-0}) and (\ref{4-5-1}), we apply Lemma \ref{short-1} to the geodesic segment $\gamma([a_i,b_i])$. Then for all $0\leq i \leq k-1$,
\begin{eqnarray}\label{4-6-2}
&&|\sqrt{\ell_{\alpha_i}(\gamma(a_i))}-\sqrt{\ell_{\alpha_i}(\gamma(b_i))}|=|\int_{a_i}^{b_i} \langle\grad \ell_{\alpha_i}^{\frac{1}{2}}(\gamma(t)), \gamma'(t)\rangle_{wp}dt|\\
&\leq&\int_{a_i}^{b_i} ||\grad \ell_{\alpha_i}^{\frac{1}{2}}(\gamma(t))||_{wp}dt\leq \sqrt{c}\cdot (b_i-a_i)\nonumber
\end{eqnarray}
where $||\cdot||_{wp}$ is the \wep norm.

Combine inequalities (\ref{4-6-1}) and (\ref{4-6-2}) we get
\begin{eqnarray*}
|\sqrt{\lsys(X)}-\sqrt{\lsys(Y)}|&\leq& \sum_{i=0}^{k-1}C(1)\cdot(t_{i+1}-t_i+a_i-b_i)+ \sum_{i=0}^{k-1} \sqrt{c}\cdot (b_i-a_i)\\
&\leq& \max\{C(1),\sqrt{c}\} \cdot (t_{k}-t_0)\\
&=&\max\{C(1),\sqrt{c}\} \cdot \dist_{wp}(X,Y).
\end{eqnarray*}

Then the conclusion follows by choosing $K=\max\{C(1),\sqrt{c}\}$.
\end{proof}

\begin{remark}
For the case $(g,n)=(1,1)$ or $(0,4)$, we let $\alpha, \beta \subset S_{g,n}$ be any two essential simple closed curves which fill the surface $S_{g,n}$.  The strata $\sT_{\alpha}$ and $\sT_{\beta}$ are two single points. By \cite{DW03, Wolpert03, Yamada04} the \wep geodesic $I$ joining $\sT_{\alpha}$ and $\sT_{\beta}$ is contained in $\Teich(S_{g,n})$ except the two end points. The Collar Lemma \cite{Keen74} implies that there exists at least one point $Z\in I$ such that $\lsys(Z)\geq 2\arcsinh{1}$. Then Theorem \ref{mt-lip} gives that $\ell(I)= \dist_{WP}(Z, \sT_{\alpha})+ \dist_{WP}(Z, \sT_{\beta})\geq \frac{2\sqrt{2\arcsinh{1}}}{K}>0$. One can see \cite[Corollary 22]{Wolpert03} for a more general statement, and see Theorem 1.7 in \cite{BB2014} for a more explicit lower bound. Since the completion $\overline{\sM_{g,n}}$ contains $\sM_{0,2g+n}$ as a totally geodesic subspace, up to a uniform multiplicative constant the quantity $\sqrt{g}$ serves as a lower bound for the diameter $\diam(\sM_{g,n})$ for large genus, as observed in Proposition 5.1 in \cite{CP12}. 
\end{remark}


\section{Proofs of Theorems \ref{mt-g}, \ref{mt-n} and \ref{leaf}}\label{g-n}
In this section we will first prove Theorem \ref{leaf} and then apply Theorem \ref{mt-lip} to finish the proofs of Theorems \ref{mt-g} and \ref{mt-n}. 

\begin{proof}[Proof of Theorem \ref{leaf}]
For the lower bound, by the Mumford compactness theorem we may assume that $X \in \partial \sM^{\geq s}_{g,n}$ and $Y \in \partial \sM^{\geq t}_{g,n}$ such that 
\[\dist_{wp}(\partial \sM^{\geq s}_{g,n},\partial \sM^{\geq t}_{g,n})=\dist_{wp}(X,Y).\]

From Theorem \ref{mt-lip} we know that $\dist_{wp}(X,Y)\geq K(\sqrt{s}-\sqrt{t})$. Thus,
\begin{eqnarray*}
\dist_{wp}(\partial \sM^{\geq s}_{g,n},\partial \sM^{\geq t}_{g,n}) \geq K(\sqrt{s}-\sqrt{t}).
\end{eqnarray*}\

For the upper bound, for any $X \in \partial \sM^{\geq s}_{g,n}$ we let $\alpha \subset X$ such that
\[\lsys(X)=\ell_{\alpha}(X)=s.\]

Recall that Equation (\ref{Rie-f}) (Riera's formula) tells that 
\begin{equation}\label{5.11-1}
\langle \grad \sqrt{2\pi \ell_{\alpha}},\grad \sqrt{2 \pi\ell_{\alpha}}\rangle_{wp} >1.
\end{equation}

It follows from standard ODE theory that there exists a smooth curve $\gamma$ of arc-length parameter $r$ in $\sM_{g,n}$ such that
\[\gamma(0)=X  \ \textit{and} \  \gamma'(r)=- \frac{\grad \sqrt{2\pi \ell_{\alpha}(\gamma(r))}}{||\grad \sqrt{2\pi \ell_{\alpha}(\gamma(r))}||_{wp}}.\]

The length function $\ell_{\alpha}$ is decreasing along $\gamma$ because for $r_1>r_2>0$,
\begin{eqnarray*}
\sqrt{2\pi \ell_{\alpha}(\gamma(r_1))}-\sqrt{2\pi \ell_{\alpha}(\gamma(r_2))}&=&\int_{r_2}^{r_1} \langle \grad \sqrt{2\pi \ell_{\alpha}(\gamma(r))}, \gamma'(t) \rangle_{wp} dr \\
 &=& - \int_{r_2}^{r_1} ||\grad \sqrt{2\pi \ell_{\alpha}(\gamma(r))}||_{wp} dr \\
&<& 0.
\end{eqnarray*}

By the inequality above we know that the curve $\gamma$ will go to the stratum whose pinching curve is $\alpha$. Since $s>t\geq 0$ and 
$\ell_{\alpha}(\gamma(0))=s$, we may assume that $r_0>0$ is a constant such that $$\ell_{\alpha}(\gamma(r_0))=t.$$ 

Then we have
\begin{eqnarray*}
\sqrt{2\pi s}-\sqrt{2\pi t}&=& \sqrt{2\pi \ell_{\alpha}(\gamma(0))}-\sqrt{2\pi \ell_{\alpha}(\gamma(r_0))} \\
&=&\int_{r_0}^{0} \langle \grad \sqrt{2\pi \ell_{\alpha}(\gamma(r))}, \gamma'(t) \rangle_{wp} dr \\
 &=& \int_{0}^{r_0} ||\grad \sqrt{2\pi \ell_{\alpha}(\gamma(r))}||_{wp} dr \\
&\geq & r_0 \quad \quad (by \ Equation \ (\ref{5.11-1})) \\
&\geq & \dist_{wp} (X, \gamma(r_0))
\end{eqnarray*}
where the last inequality uses the fact that $\gamma$ uses the arc-length parameter.

Since $\ell_{\alpha} (\gamma(r_0))=t<s=\ell_{\alpha} (X)$, the \wep geodesic joining $X$ and $\gamma(r_0)$ will cross the leaf $\partial \sM_{g,n}^{\geq t}$. Thus,
\[\dist_{wp}(X, \gamma(r_0)) \geq \dist_{wp}(X, \partial \sM_{g,n}^{\geq t}).\]

Since $\ell_{\alpha} (X)=s$, the two inequalities above imply that
\begin{eqnarray*}
\dist_{wp}(\partial \sM^{\geq s}_{g,n},\partial \sM^{\geq t}_{g,n}) &\leq& \dist_{wp}(X, \partial \sM_{g,n}^{\geq t})\\
&\leq & \sqrt{2 \pi}(\sqrt{s}-\sqrt{t}).
\end{eqnarray*}

Then the conclusion follows by choosing 
\[K'=\max\{\sqrt{2\pi}, \frac{1}{K}\}.\]
\end{proof}

\begin{remark}
The argument in the proof of Theorem \ref{leaf} also gives that \textsl{$\max_{X \in \partial \sM^{\geq s}_{g,n}}\dist_{wp}(X,\partial \sM^{\geq t}_{g,n})$ is uniformly comparable to $(\sqrt{s}-\sqrt{t})$.}
\end{remark}

Although \tec space is non-compact, the systole function $\lsys(\cdot):\Teich(S_{g,n})\to \mathbb{R}^+$ is bounded above by a constant depending on $g$ and $n$. Follow \cite{BMP14} we define
\[sys(g,n):=\sup_{X\in \Teich(S_{g,n})}\lsys(X).\]
By Mumford's compactness theorem \cite{Mumford71} this supremum is in fact a maximum. We list some bounds for $sys(g,n)$ which will be useful in the proofs of Theorems \ref{mt-g} and \ref{mt-n}.  
One can see \cite{BMP14} for more details on $sys(g,n)$.

We always assume that $3g+n-3>0$. Since the set of shortest closed geodesics of a maximal surface fills the surface, the Collar Lemma \cite{Keen74} gives that
\begin{eqnarray}\label{5-1}
\ sys(g,n)\geq 2 \arcsinh{1}.
\end{eqnarray}

Buser and Sarnak proved in \cite{BS94} that there exists a universal constant $U>0$ such that $sys(g,0)\geq U \ln{g}$. And actually they also proved that there exists a subsequence $\{g_k\}_{k\geq 1}$ of $\{g\}_{g\geq 1}$ such that $sys(g_k,0)\geq \frac{4}{3} \ln{g_k}$. If we allow the surface to have punctures, based on Buser-Sarnak's work, Balacheff, Makover and Parlier \cite[Prop. 2]{BMP14} proved the following lower bound which will be useful to prove Theorem \ref{mt-g}. 
\begin{eqnarray}\label{5-2}
\ sys(g,n)\geq \min\{U \ln{g},2 \arccosh{(\frac{2(g-1)}{n}+1)}\}.
\end{eqnarray} 

An interesting upper bound for $sys(g,n)$ was provided by Schmutz in \cite{Schmutz94}, which says that if $n\geq 2$, $sys(g,n)\leq 4 \arccosh{(\frac{6g-6+3n}{n})}$. If $g\geq 1$, Part 1 of \cite[Theorem]{BMP14} tells that $sys(g,0)<sys(g,1)<sys(g,2)$. Thus, these two results give that for all $g,n$ with $3g+n-3>0$,
\begin{eqnarray}\label{5-3}
\ sys(g,n)\leq \min\{4 \arccosh{(3(g+1))},4 \arccosh{(\frac{6g-6+3n}{n})}\}.
\end{eqnarray} 

Now we are ready to prove Theorems \ref{mt-g} and \ref{mt-n}.
\begin{proof}[Proof of Theorem \ref{mt-g}]
For any $X \in \sM_{g,n}$, we let $\alpha \subset X$ be a systolic curve, i.e., $\ell_{\alpha}(X)=\lsys(X)$. Inequality (\ref{5-3}) tells that if $g\geq 2$,
\begin{eqnarray}\label{5-1-1}
\lsys(X)\leq 4 \arccosh{(4g)}.
\end{eqnarray}

Let $\sT_{\alpha}\subset \overline{\sM}_{g,n}$ be the stratum whose vanishing curve is $\alpha$. Then we have for all $g\geq 2$,
\begin{eqnarray*}
\dist_{wp}(X,\partial \overline{\sM}_{g,n} ) &\leq& \dist_{wp}(X,\sT_{\alpha_{g,n}} )\\
&\leq& \sqrt{2\pi \ell_{\alpha}(X)}, \quad (by \ Theorem \ (\ref{d-stra})) \\
&=&  \sqrt{2\pi \lsys(X)}\\
&\leq &  \sqrt{2\pi \cdot 4 \arccosh{(4g)}} \quad (by \ inequality \ (\ref{5-1-1}))\\
&< & \sqrt{32\pi} \cdot \sqrt{\ln{g}}.
\end{eqnarray*}

Since $X\in \sM_{g,n}$ is arbitrary, we have
\[\InR(\sM_{g,n})\leq \sqrt{32\pi} \cdot \sqrt{\ln{g}}.\]\

For the lower bound, from inequality (\ref{5-2}) one may choose a surface $Y\in \sM_{g,n}$ such that 
\begin{eqnarray}\label{5-2-1}
\lsys(Y)\geq\min\{U \ln{g},2 \arccosh{(\frac{2(g-1)}{n}+1)}\}.
\end{eqnarray} 

Thus, there exists a constant $k(n)$, only depending on $n$, such that
\[\lsys(Y)\geq k(n)\cdot \ln{g}.\]

We let $Z\in \partial{ \sM_{g,n}}$ such that 
$$\dist_{wp}(Y,Z)=\dist_{wp}(Y,\partial{\sM_{g,n}}).$$

Then we have
 \begin{eqnarray*}
\InR(\sM_{g,n})&\geq& \dist_{wp}(Y,\partial{\sM_{g,n}}) \\
&=&\dist_{wp}(Y,Z)\\
&\geq& \frac{1}{K} |\sqrt{\lsys(Y)}-\sqrt{\lsys(Z)}| \quad (by \ Theorem \ \ref{mt-lip}))\\
&=&\frac{1}{K} \sqrt{\lsys(Y)}  \quad (because \ \lsys(Z)=0) \\
&\geq& \frac{\sqrt{k(n)}}{K}\sqrt{\ln{g}}
\end{eqnarray*} 
where $K$ is the universal constant from Theorem \ref{mt-lip}.
\end{proof}

\begin{remark}\label{g-1}
The proof of Theorem \ref{mt-g} also leads to the following result.
\begin{theorem}
For all $g,n$ with $g\geq 2$, then
\[\InR (\Teich(S_{g,n})) \asymp_g \sqrt{\ln{g}}.\]
\end{theorem}
\end{remark}

\begin{remark}\label{n(g)}
In the proof of the lower bound of Theorem \ref{mt-g}, the quantity $2 \arccosh{(\frac{2(g-1)}{n}+1)}$ is applied. Observe that for any constant $a\in (0,1)$, the quantity $2 \arccosh{(\frac{2(g-1)}{g^a}+1)}$ is comparable to $\ln{g}$ as $g$ goes to infinity. 
So we also get that
\[\InR(\sM_{g, [g^{a}]})\asymp_g \sqrt{\ln{g}}.\]
\end{remark}

The proof of Theorem \ref{mt-n} is similar to the one of Theorem \ref{mt-g}. 
\begin{proof}[Proof of Theorem \ref{mt-n}]

For any $X \in \sM_{g,n}$, we let $\alpha \subset X$ be a systolic curve, i.e., $\ell_{\alpha}(X)=\lsys(X)$. From inequality (\ref{5-3}) we know that there exists a constant $d(g)>0$, only depending on $g$, such that for all $n\geq 4$,
\begin{eqnarray}\label{5-3-1}
sys(g,n)\leq d(g).
\end{eqnarray}

Let $\sT_{\alpha}\subset \overline{\sM}_{g,n}$ be the stratum whose vanishing curve is $\alpha$. Then we have for all $n\geq 4$,
\begin{eqnarray*}
\dist_{wp}(X,\partial \overline{\sM}_{g,n} ) &\leq& \dist_{wp}(X,\sT_{\alpha_{g,n}} )\\
&\leq& \sqrt{2\pi \ell_{\alpha}(X)}, \quad (by \ Theorem \ (\ref{d-stra})) \\
&=&  \sqrt{2\pi \lsys(X)}\\
&\leq&  \sqrt{2\pi \cdot d(g)} \quad (by \ inequality \ (\ref{5-3-1})).
\end{eqnarray*}

Since $X\in \sM_{g,n}$ is arbitrary, we have
\[\InR(\sM_{g,n})\leq \sqrt{2\pi \cdot d(g)}.\]\

For the lower bound, we will give two different proofs: the first one will apply Theorem \ref{mt-lip}, and the other one will apply Lemma \ref{short-1} instead of Theorem \ref{mt-lip}.  

Method (1): we apply Theorem \ref{mt-lip}. First from inequality (\ref{5-1}) one may choose a surface $Y\in \sM_{g,n}$ such that 
\begin{eqnarray}\label{5-4-1}
\lsys(Y)\geq 2 \arcsinh{1}.
\end{eqnarray} 

We let $Z\in \partial{ \sM_{g,n}}$ such that 
$$\dist_{wp}(Y,Z)=\dist_{wp}(Y,\partial{\sM_{g,n}}).$$

Then we have
 \begin{eqnarray*}
\InR(\sM_{g,n})&\geq& \dist_{wp}(Y,\partial{\sM_{g,n}}) \\
&=&\dist_{wp}(Y,Z)\\
&\geq& \frac{1}{K} |\sqrt{\lsys(Y)}-\sqrt{\lsys(Z)}| \quad (by \ Theorem \ \ref{mt-lip}))\\
&=&\frac{1}{K} \sqrt{\lsys(Y)}  \quad (because \ \lsys(Z)=0) \\
&\geq& \frac{ \sqrt{2 \arcsinh{1}}}{K}
\end{eqnarray*} 
where $K$ is the universal constant from Theorem \ref{mt-lip}.\\

Method (2): we apply Lemma \ref{short-1} without using Theorem \ref{mt-lip}. Similarly from inequality (\ref{5-1}) one may choose a surface $Y\in \sM_{g,n}$ such that 
\begin{eqnarray}\label{5-4-1-1}
\lsys(Y)\geq 2 \arcsinh{1}.
\end{eqnarray} 

We let $Z\in \partial{ \sM_{g,n}}$ such that 
$$\dist_{wp}(Y,Z)=\dist_{wp}(Y,\partial{\sM_{g,n}}).$$

Let $\alpha \subset S_{g,n}$ be a pinched curve on $Z$, i.e., $\ell_{\alpha}(Z)=0$. Consider the shortest \wep geodesic $\gamma:[0,s]\to \overline{\sM}_{g,n}$ such that $\gamma(0)=Y$ and $\gamma(s)=Z$ where $s=\dist_{wp}(Y,Z)$. Since $\ell_{\alpha}(Z)=0$,  the constant $$s_{0}:=\inf\{t_0\in [0,s];\ \ell_{\alpha}(\gamma(t))\leq 1, \ \forall t_0\leq t\leq s\}$$ is well-defined. Since $2\arcsinh{1}\geq 1$, from inequality (\ref{5-4-1-1}) and the definition of $s_0$ we have $$\ell_{\alpha}(\gamma(s_0))=1.$$ 

We apply Lemma \ref{short-1} to the geodesic $\gamma([t_0,s))$. Then,
\begin{eqnarray*}
1&=&|\sqrt{\ell_{\alpha}(\gamma(s_0))}-\sqrt{\ell_{\alpha}(\gamma(s))}|\\
&=&|\int_{s_0}^{s} \langle\grad \ell_{\alpha}^{\frac{1}{2}}(\gamma(t)), \gamma'(t)\rangle_{wp}dt|\\
&\leq&\int_{s_0}^{s} ||\grad \ell_{\alpha}^{\frac{1}{2}}(\gamma(t))||_{wp}dt.
\end{eqnarray*}

Since $\ell_{\alpha}(\gamma(t))\leq 1$ for all $s_0\leq t \leq s$, from Lemma \ref{short-1} we have
\begin{eqnarray*}
1\leq \sqrt{c}\cdot (s-s_0)\leq \sqrt{c}\cdot \dist_{wp}(Y,Z)
\end{eqnarray*}
where $c$ is the constant in Lemma \ref{short-1}.

Thus,
 \begin{eqnarray*}
\InR(\sM_{g,n})&\geq& \dist_{wp}(Y,\partial{\sM_{g,n}}) \\
&=&\dist_{wp}(Y,Z)\\
&\geq& \frac{1}{\sqrt{c}}.
\end{eqnarray*}

The positive lower bounds from the two methods above are different. But both of them are independent of $g$ and $n$.\\

The proof is complete.
\end{proof}

\begin{remark}\label{n-1}
The proof of Theorem \ref{mt-n} also leads to 
\[\InR (\Teich(S_{g,n}))  \asymp_n 1.\]
\end{remark}

\begin{remark}\label{g(n)}
In the proof above, the quantity $4 \arccosh{(\frac{6g-6+3n}{n})}$ is applied to establish the upper bound. Observe that for any constant $a\in (0,1)$, $4 \arccosh{(\frac{6n^a-6+3n}{n})}$ is comparable to $1$ as $n$ goes to infinity. Actually the proof of Theorem \ref{mt-n} also yields that 
\[\InR(\sM_{[n^{a}],n})\asymp_n 1.\]
\end{remark}


\section{\wep volume for large genus}\label{wp-v}
For simplicity, we will focus on \tec space of closed surfaces endowed with the \wep metric, which is denoted by $\Teich(S_g)$. The results in this section are still true for surfaces with punctures. The space $\Teich(S_g)$ is incomplete \cite{Chu76, Wolpert75}, negatively curved \cite{Tromba86, Wolpert86} and uniquely geodesically convex \cite{Wolpert87}. We will study the asymptotic behavior of the \wep volumes of geodesic balls of finite radii in $\Teich(S_g)$ as the genus $g$ goes to infinity. The main goal in this section is to prove Theorem \ref{decay-0}.

The proof of Theorem \ref{decay-0} involves using Theorem \ref{mt-g} together with the following theorem due to Teo \cite{Teo09} on the Ricci curvature on the thick-part of the  \tec space. Let $\epsilon_0>0$. Recall that
$\Teich(S_{g,n})^{\geq \epsilon_0}$ is the $\epsilon_0$-thick part $\sT(S_{g,n})^{\geq \epsilon_0}$ endowed with the \wep metric.
\begin{theorem}\cite[Prop. 3.3]{Teo09}\label{teo}
The Ricci curvature of $\Teich(S_{g,n})^{\geq \epsilon_0}$ is bounded from below by $-C'(\epsilon_0)$ where $C'(\epsilon_0)>0$ is a constant which only depends on $\epsilon_0$.
\end{theorem}

The constant $C'(\epsilon_0)$ above roughly behaves like $\frac{2}{\pi \epsilon_0^2}$ as $\epsilon_0$ goes to $0$. 

Huang \cite{Huang07-a} showed that the \wep sectional curvature is not bounded below by any negative constant. For suitable choice of $\epsilon_0>0$, in \cite{WW15} it was shown that the minimal \wep sectional curvature over $\Teich(S_{g,n})^{\geq \epsilon_0}$ is comparable to $-1$ even as $g$ goes to infinity. For the most recent developments on the \wep curvature on the thick part of \tec space, one may refer to \cite{Huang07, WW15, Wu15-curvature}.

Since the completion $\overline{\Teich(S_{g})}$ of $\Teich(S_g)$ is not locally compact \cite{Wolpert03}, the \wep volume of a geodesic ball of finite radius in $\Teich(S_g)$ may blow up. The following result is well-known to experts. We provide it here for completeness.
\begin{proposition}\label{blow-up}
Let $X_g \in \Teich(S_g)$. Then, for any positive constant $r$ with $r>\dist_{wp}(X_g, \partial \overline{\Teich(S_{g})})$ the \wep volume satisfies 
\[\Vol_{wp}(B(X_g;r))=\infty\]
where $B(X_g;r)=\{Y\in \Teich(S_{g}); \ \dist_{wp}(Y,X_g)<r\}$.
\end{proposition}

\begin{proof}
Let $s=\dist_{wp}(X_g, \partial \overline{\Teich(S_{g})})<r$ and $\gamma: [0,s]\to \overline{\Teich(S_{g})}$ be the \wep geodesic such that $\gamma(0)=X_g$ and $\gamma(s)\in \partial \overline{\Teich(S_{g})}$. By results in \cite{DW03, Wolpert03,Yamada04} we know that the image satisfies
\[\gamma([0,s))\subset \Teich(S_{g,n}).\] 
Since $\gamma(s)\in \partial \overline{\Teich(S_{g})}$, we may assume that $\gamma(s)\in \sT_{\sigma}$ where $\sT_{\sigma}$ is some stratum. Let $\tau_{\sigma}=\Pi_{\alpha \subset \sigma^0}\tau_{\alpha}$ be the Dehn-twist on the multi curves in $\sigma^0$. Take a number $0<\epsilon<\frac{r-s}{2}$. Since the mapping class group acts properly discontinuously on $\Teich(S_{g})$ \cite{IT92}, there exists a positive constant $\epsilon'<\epsilon$ such that  the geodesic balls $\{\tau_{\sigma}^k \circ B(\gamma(s-\epsilon);\epsilon')\}_{k\geq 0}$ are pairwise disjoint. It is clear that $\tau_{\sigma}^k \circ \gamma(s)=\gamma(s)$ and $\tau_{\sigma}^k \circ B(\gamma(s-\epsilon);\epsilon')=B(\tau_{\sigma}^k \circ \gamma(s-\epsilon);\epsilon')$ for all $k\geq 0$. Then, for any $k\geq 0$ and $Z \in  B(\tau_{\sigma}^k \circ\gamma(s-\epsilon);\epsilon')$, the triangle inequality tells that 
\begin{eqnarray*}
\dist_{wp}(Z,X_g)&\leq& \dist_{wp}(Z,\tau_{\sigma}^k \circ\gamma(s-\epsilon))\\
&+&\dist_{wp}(\tau_{\sigma}^k \circ\gamma(s-\epsilon),\gamma(s))+\dist_{wp}(\gamma(s),X_g)\\
&<& \epsilon'+\epsilon+s\\
&<& 2\epsilon+s\\
&<&r.
\end{eqnarray*}
That is, for all $k\geq 0$, $\tau_{\sigma}^k \circ B(\gamma(s-\epsilon);\epsilon')
\subset B(X_g;r)$. Since $\{\tau_{\sigma}^k \circ B(\gamma(s-\epsilon);\epsilon')\}_{k\geq 0}$ are pairwise disjoint, we have
\begin{eqnarray*}
\Vol_{wp}(B(X_g;r))&\geq& \Vol_{wp}(\cup_{k\geq0}\tau_{\sigma}^k \circ B(\gamma(s-\epsilon);\epsilon'))\\
&=& \sum_{k\geq 0}\Vol_{wp}(\tau_{\sigma}^k \circ B(\gamma(s-\epsilon);\epsilon')) \\
&=&\infty
\end{eqnarray*}
where in the last step we use that fact $\tau_{\sigma}$ is an isometry on $\Teich(S_{g})$.
\end{proof}

Let $\{X_g\}_{g\geq 2}$ be a sequence of points in \tec space and $\{r_g\}_{g\geq 2}$ be a sequence of positive numbers. In this section we will study the asymptotic behavior of $\{\Vol_{wp}(B(X_g;r_g))\}_{g\geq 2}$ as $g$ tends to infinity. In light of Proposition \ref{blow-up}, we need to assume that the completions $\{\overline{B(X_g;r_g)}\}_{g\geq 2} \subset \overline{\Teich(S_{g})}$ always do not intersect the boundary of \tec space. For any $r_0>0$, we define $\sU(\Teich(S_{g}))^{\geq r_0}$ to be the subset in $\Teich(S_g)$ which is at least $r_0$-away from the boundary. More precisely,
\begin{eqnarray*}
\sU(\Teich(S_{g}))^{\geq r_0}:=\{X_g \in \Teich(S_g); \dist_{wp}(X_g; \partial \overline{\Teich(S_g)}\geq r_0\}.
\end{eqnarray*}

Theorems \ref{mt-g} and \ref{d-stra} tell that the largest radius of the geodesic ball in the set $\sU(\Teich(S_{g}))^{\geq r_0}$ is comparable to $\sqrt{\ln{g}}$ as $g$ goes infinity. 

Before we prove Theorem \ref{decay-0}, we first provide a lemma which says that the set $\sU(\Teich(S_{g}))^{\geq r_0}$
is contained in some thick part of \tec space. More precisely,
\begin{lemma}\label{6-1-thick}
For any $r_0>0$, there exists a constant $\epsilon(r_0)$, only depending on $r_0$, such that 
\[\sU(\Teich(S_{g}))^{\geq r_0} \subset \Teich(S_g)^{\geq \epsilon(r_0)}.\]
\end{lemma}

\begin{proof}
The proof is a direct application of Theorem \ref{d-stra}. For any $X_g \in \sU(\Teich(S_{g}))^{\geq r_0}$ we let $\alpha_g \subset X_g$ be an essential simple closed curve such that $\ell_{\alpha_g}(X_g)=\lsys(X_g)$, and $\sT_{\alpha}$ be the stratum in $\overline{\Teich(S_g)}$ whose vanishing curve is $\alpha$. Then, by Theorem \ref{d-stra} we have
\begin{eqnarray*}
r_0 &\leq& \dist_{wp}(X_g, \sT_{\alpha_g})\\
&\leq & \sqrt{2\pi \lsys(X_g)}.
\end{eqnarray*}

Thus, 
\[X_g \in \Teich(S_g)^{\geq \frac{r_0^2}{4\pi^2}}.\] 

Then the conclusion follows by choosing 
\[\epsilon(r_0)= \frac{r_0^2}{4\pi^2}.\]
\end{proof}

Now we are ready to prove Theorem \ref{decay-0}.

\begin{proof}[Proof of Theorem \ref{decay-0}]
Let $B(X_g;r_g)\subset \sU(\Teich(S_{g}))^{\geq r_0}$ be an arbitrary geodesic ball where $X_g \in \Teich(S_g)$ and $r_g>0$. Lemma \ref{6-1-thick} tells that there exists a constant $\epsilon(r_0)$, only depending on $r_0$, such that
\[B(X_g;r_g) \subset \Teich(S_g)^{\geq \epsilon(r_0)}.\]

By Teo's curvature bound (see Theorem \ref{teo}) there exists a constant $C'(r_0)>0$, only depending on $r_{0}$, such that the Ricci curvature satisfies 
\begin{eqnarray} 
\Ric|_{B(X_g;r_g)} &\geq& -C'(r_0)\\
&=& (6g-7)\cdot(\frac{-C'(r_0)}{6g-7}). \nonumber
\end{eqnarray}

From the Gromov-Bishop Volume Comparison Theorem \cite{Gromov1981} we have
\begin{eqnarray} 
\Vol_{wp}(B(X_g;r_g))\leq \Vol_{Euc}(\mathbb{S}^{6g-7})\int_{0}^{r_g}(\frac{\sinh{(\sqrt{\frac{C'(r_0)}{6g-7}}t)}}{\sqrt{\frac{C'(r_0)}{6g-7}}})^{6g-7}dt
\end{eqnarray}
where $\Vol_{Euc}(\mathbb{S}^{6g-7})$ is the standard $(6g-7)$-dimensional volume of the unit sphere. By Stirling's formula we have
\begin{eqnarray*} 
\frac{\Vol_{Euc}(\mathbb{S}^{6g-7})}{(\sqrt{\frac{C'(r_0)}{6g-7}})^{6g-7}}&\leq& \frac{2\pi^{\frac{6g-7}{2}}}{\Gamma(\frac{6g-7}{2})}  (\frac{6g-7}{C'(r_0)})^{\frac{6g-7}{2}}\\
&\leq & 2\pi^{\frac{6g-7}{2}} (\frac{2}{6g-9})^{3g-\frac{9}{2}}(\frac{6g-7}{C'(r_0)})^{\frac{6g-7}{2}}\\
&\leq & C^{g}
\end{eqnarray*}
for some constant $C>0$. Thus,
\begin{eqnarray} 
\Vol_{wp}(B(X_g;r_g))\leq C^g \int_{0}^{r_g}(\sinh{(\sqrt{\frac{C'(r_0)}{6g-7}}t)})^{6g-7}dt.
\end{eqnarray}

Since $B(X_g;r_g)\subset \sU(\Teich(S_{g}))^{\geq r_0}$, Theorem \ref{mt-g} (or Remark \ref{g-1}) tells that $r_g\leq \sqrt{32\pi \ln{g}}$ for all $g\geq 2$. Note that $\lim_{g\to \infty}\frac{\ln{g}}{g}=0$, thus one may assume that there exists a constant $D>0$ such that 
\[\sinh{(\sqrt{\frac{C'(r_0)}{6g-7}}t)}\leq \frac{Dt}{\sqrt{g}},\ \forall0\leq t \leq r_g.\] 

Thus,
\begin{eqnarray} 
\Vol_{wp}(B(X_g;r_g))\leq C^g \int_{0}^{r_g}(\frac{Dt}{\sqrt{g}})^{6g-7}dt.
\end{eqnarray}

Recall that $r_g\leq \sqrt{32\pi \ln{g}}$. A direct computation gives that
\begin{eqnarray} 
\Vol_{wp}(B(X_g;r_g))\leq E^g \frac{(\ln{g})^g}{g^{3g}}
\end{eqnarray}
for some constant $E>0$. Observe that for any $\epsilon>0$, 
\[\displaystyle \lim_{g\to \infty}\frac{E^g \frac{(\ln{g})^g}{g^{3g}}}{(\frac{1}{g})^{(3-\epsilon)g}}=0.\]

Then, there exists a constant $F>0$ such that
\begin{eqnarray} 
\Vol_{wp}(B(X_g;r_g))\leq F\cdot (\frac{1}{g})^{(3-\frac{\epsilon}{2})g} .
\end{eqnarray}

Since the geodesic ball $B(X_g;r_g)\subset \sU(\Teich(S_{g}))^{\geq r_0}$ is arbitrary, the conclusion follows.
\end{proof}

\begin{proof}[Proof of Corollary \ref{R-0}]
Let $r_0=1$ in Theorem \ref{decay-0}. For any fixed constant $R>0$, by Theorem \ref{decay-0} it suffices to show that there exists a constant $\epsilon(R)>0$ such that
\begin{eqnarray}\label{7-1} 
B(X_g; R)\subset \sU(\Teich(S_{g}))^{\geq 1}, \ \forall X_g \in \sU(\Teich(S_{g}))^{\geq \epsilon(R)}.
\end{eqnarray}

We choose $\epsilon(R)=R+1$. The triangle inequality tells that for all $Y \in B(X_g; R)$,
\begin{eqnarray*}
\dist_{wp}(Y, \partial(\Teich(S_{g,n}))&\geq& \dist_{wp}(X_g, \partial(\Teich(S_{g,n})))-\dist_{wp}(Y, X_g)\\
&\geq& \epsilon(R)-R\\
&=&1.
\end{eqnarray*}

Then Equation (\ref{7-1}) follows since $Y \in B(X_g; R)$ is arbitrary.
\end{proof}

\begin{remark}
Theorem 4.2 in \cite{Mirz13} tells that the \wep volume of moduli space $\sM_g$ is concentrated in the thick part as the genus $g$ tends to infinity, which blows up rapidly. Theorem \ref{decay-0} says that the \wep volume of any \wep geodesic ball in the thick part of moduli space will decay to $0$ as $g$ tends to infinity. It would be very \textsl{interesting} to study the asymptotic shape of $\sM_g$ as $g$ tends to infinity. 
\end{remark}

\bibliographystyle{amsalpha}
\bibliography{ref}

\end{document}